\documentclass{amsart}

\usepackage{amsmath}
\usepackage{amsfonts}
\usepackage{amssymb}
\usepackage{amsthm}

\usepackage{epsfig,epstopdf}
\usepackage[dvipsnames,usenames]{color}
\usepackage{hyperref}
\usepackage{xifthen}
\usepackage{graphicx}
\usepackage{tikz}
\tikzstyle{vertex}=[circle, draw, inner sep=0pt, minimum size=6pt]

\usepackage[
top    = 1in,
bottom = 1in,
left   = 1in,
right  = 1in]{geometry}
\usepackage{graphicx}

\tikzset{
  treenode/.style = {align=center, inner sep=0pt, text centered,
    font=\sffamily},
  arn_n/.style = {treenode, circle, black, text width=1em},
  arn_r/.style = {treenode, circle, blue, draw=blue, 
    text width=1em, very thick}}

\theoremstyle{plain}
\newtheorem{theorem}{Theorem}[section]
\newtheorem{lemma}[theorem]{Lemma}
\newtheorem{corollary}[theorem]{Corollary}
\newtheorem{prop}[theorem]{Proposition}

\theoremstyle{definition}

\newtheorem{defn}[theorem]{Definition}
\newtheorem{ex}[theorem]{Example}

\newcommand{\C}{\mathbb{C}}

\newcommand{\N}{\mathbb{N}}
\newcommand{\R}{\mathbb{R}}

\newcommand{\ds}{\displaystyle}

\newcommand{\dist}{\operatorname{d}}

\renewcommand{\mid}{\, : \, }

\newcommand\Lhat{\Hat{\mathcal{L}}} 
 
\newcommand\Lcerohat{\Hat{\mathcal{L}_{0}}} 
\newcommand\LAcerohat{\Hat{\mathcal{L}^{a}_{0}}} 
\newcommand\supfvw{\sup\limits_{v \sim w }{ \lvert f(v)-f(w)} \rvert  }

\newcommand\maxfNv{\max\limits_{w \in N_v }{ \lvert f(v)-f(w)} \rvert  }
\newcommand\supfL{\sup\limits_{f \in \Lhat }\{|f(v)| : f(a)=0, \fnorm \leq 1 \}}
\newcommand\supfLcero{\sup\limits_{f \in \Lcerohat }\{|f(v)| : f(a)=0, \fnorm \leq 1 \}}

\newcommand\maxgNv{\max\limits_{w \in N_v }{ \lvert g(v)-g(w)} \rvert  }
\newcommand\supgvw{\sup\limits_{v \sim w }{ \lvert g(v)-g(w)} \rvert  }

\newcommand\suppsifvw{\sup\limits_{v \sim w} |\psi(v)f(v)-\psi(w)f(w)|}
\newcommand\blanknorm{\| \cdot \|_{\Lhat}^{a}}
\newcommand\fnorm{\| f \|_{\Lhat}^{a}}
\newcommand\gnorm{\| g \|_{\Lhat}^{a}}
\newcommand\fplusgnorm{\| f+g \|_{\Lhat}^{a}}
\newcommand\fnormdef{|f(a)|+\supfvw}

\newcommand\sigmapsidef{\sup\limits_{v\sim w}\dist(a,v)|\psi(v)-\psi(w)|}
\newcommand\psinorm{\|\psi\|_{\Lhat}^{a}}
\newcommand\psisup{\|\psi\|_{\infty}}

\begin{document}

\title{Multiplication Operators on the Lipschitz Space of an Infinite Graph}
\author{Jos\'e A. Issa-Barbar\'a}
\address{Washington University in St. Louis (WashU), MO, USA}
\email{i.jose@wustl.edu}
\author{Rub\'en A. Mart\'inez-Avenda\~no}
\address{Instituto Tecnol\'ogico Aut\'nomo de M\'exico (ITAM), Mexico City, Mexico}
\email{ruben.martinez.avendano@gmail.com}


\thanks{The second author is partially supported by  the Asociaci\'on Mexicana de Cultura A.C}

\begin{abstract}
    The Lipschitz space of an infinite (locally-finite) graph is defined as the set of functions on the vertices of the graph such that the differences of the values between adjacent vertices remain bounded. In this paper we prove that this set is a Banach space when endowed with its natural norm, and we define the little Lipschitz space as the subspace where these differences tend to zero. We consider the multiplication operators on these spaces and characterize their boundedness, compactness and the spectra. We also obtain estimates of the norm and essential norm, and we characterize when these operators are isometric.
\end{abstract}

\maketitle

\section{Introduction}

The study of spaces of functions on infinite trees was initiated more than 50 years ago, probably by Cartier in the papers \cite{Cartier1,Cartier2}; his motivation was to study harmonic functions on the unit disk and see what information the functions on the tree gave on the harmonic functions on the unit disk. Nevertheless, the first study of these spaces as Banach spaces is more recent, and was given a great impulse by several papers of Cohen and Colonna (see, for example, \cite{CoCo1,CoCo2}). Furthermore, in the seminal paper \cite{CoEa}, Colonna and Easley introduced the Lipschitz space of a tree as the set of complex-valued functions defined on the vertices of the tree such that they are Lipschitz when the tree is given the metric induced by the distance between vertices. In their paper, they showed that this is in fact a Banach space when endowed with a natural norm. They also introduced the little Lipschitz space of a tree as the subspace of functions such that the differences of the values of the functions in adjacent vertices vanishes far from the root (we will give the precise definition in Section 2) and they show that this is a closed subspace of the Lipschitz space.

The study of multiplication operators on Banach spaces of functions is perhaps almost a century old and we will not try to give a comprehensive account here. Nevertheless, we should mention that this class of operators is the one that is usually first studied when one encounters a Banach space of functions and this class of operators usually offers many examples and counterexamples to theorems and conjectures one has on these spaces. Naturally, Colonna and Easley studied this type of operators in their paper. In fact, they characterize the bounded multiplication operators on these spaces, the compact multiplication on these spaces, they obtain estimates for the norms and essential norms, they find the spectra of these operators and they completely characterize isometric multiplication operators. We refer the reader to the Colonna and Easley paper \cite{CoEa} for an explanation of their motivation for the study of theses spaces and the multiplication operators on them, as well as a historical overview of the study of multiplication and composition operators (see also \cite{AlCoEa}).

In the present paper, we generalize the results of the paper \cite{CoEa} to infinite and locally-finite connected graphs. Although most of the results proven in \cite{CoEa} carry to the setting of general graphs, there is an added difficulty here since, once we fix a vertex to serve as the ``root'' of the graph, a vertex may have more than one ancestor, and furthermore, there may be other neighbors with the same distance to the ``root''. We offer a definition which we believe is appropriate and which allows us to prove essentially all of the results of \cite{CoEa}; in most of them essentially the same proofs work, once we adapt them to our definition, but in a few cases there were nontrivial difficulties. We have chosen to include most of the proofs here to make the paper self-contained.

After we had completed this work, Prof.~Flavia Colonna kindly pointed out to us the unpublished paper \cite{CoLo} (partly based on Rachel Locke's Ph.D.~thesis \cite{Locke}) in which they also study the Lipschitz space of a graph, but use a different approach for some of the proofs. We invite the reader to compare the results to gain a deeper understanding of the subject.

The organization of our paper is as follows. In Section 2 we give the basic definitions and we prove the basic facts about the Lipschitz space. In Section 3 we define the little Lipschitz space and we prove its basic properties. In Section 4 we study the weak and strong convergence of functions on these spaces. In Section 5, we characterize the boundedness of the multiplication operators on both the Lipschitz and little Lipschitz spaces and we offer some estimates for the norms. In Section 6 we characterize the spectra of the multiplication operators on these spaces and in Section 7 we characterize the compactness of these operators. Later, in Section 8, we give estimates for the essential norm of these operators. Lastly, in Section 9 we show that multiplication operators are isometries only in the most trivial of cases.

\section{The Lipschitz Space of an Infinite Graph}

For the purposes of this paper, a {\em graph} $G=(V,E)$ will always be a connected undirected simple graph, the set of vertices will be infinite, and every vertex will have finite degree ({\it i.e.}, $G$ is locally-finite). Observe that this implies that the set of vertices is countably infinite. We use the notation $v \sim w$ to denote that two vertices $v$ and $w$ are adjacent. For ease of notation, we will use the letter $G$ to denote the set of vertices of the graph, instead of $V$. We denote by $\dist(v,w)$ the (natural) distance between the vertices $v$ and $w$ in $G$.

Let us start by defining the space of functions in which we will be working throughout this paper. 

\begin{defn} 
Let $G=(V,E)$ be a graph. Then 
\begin{equation*}
\Lhat:=\{f:G \to \C  \mid   \sup\limits_{v \sim w } { |f(v)-f(w)}|  <\infty\}
\end{equation*}
where we write
\begin{equation*}
\supfvw:=\sup\{ |f(v)-f(w)| : v,w\in G \text{ and } v \sim w \},
\end{equation*}
for short.
\end{defn}

The space $\Lhat$ is, in fact, the space of complex-valued Lipschitz functions on $G$. Indeed, observe that for vertices $x\neq y \in G$, if $x=v_0 \sim v_1 \sim v_2 \sim \cdots \sim v_{n-1} \sim v_n =y$ is a path of length $n=\dist(x,y)$ between  $x$ and $y$, we have
\[
    \frac{|f(x)-f(y)|}{\dist(x,y)} 
    \leq \frac{1}{n} \sum_{j=0}^{n-1} |f(v_j) - f(v_{j+1})| \leq \supfvw,
\]
and hence
\[
\sup_{v\neq w} \frac{|f(v)-f(w)|}{\dist(v,w)} \leq \supfvw.
\]
It is obvious that for $x \sim y \in G$
\[
|f(x) -f(y) | = \frac{|f(x) -f(y) |}{\dist(x,y)} \leq \sup_{v\neq w} \frac{|f(v)-f(w)|}{\dist(v,w)},
\]
and hence
\[
\sup_{v\neq w} \frac{|f(v)-f(w)|}{\dist(v,w)} = \supfvw,
\]
showing that $\Lhat$ is indeed the space of Lipschitz functions on $G$.

As it was the case in \cite{CoEa}, the functions in $\Lhat$ are not necessarily bounded. Indeed, for a fixed vertex $a$, the function $f:G\to \C$ defined by $f(v)=\dist(a,v)$ for all $v \in G$ is not bounded, but it belongs to $\Lhat$. Nevertheless, notice that if $f:G\to \C$ is a bounded function then  $f\in \Lhat$.

It is straightforward to check that $\Lhat$ is a vector space, since clearly
\begin{equation*}
\sup\limits_{v \sim w } |(f+\lambda g)(v) - (f+\lambda g)(w)| \leq  \sup\limits_{v \sim w }|f(v)-f(w)| + |\lambda| \sup\limits_{v \sim w } | g(v)-g(w)|,
\end{equation*}
for $f, g \in \Lhat$ and $\lambda \in \C$.

Now, we define a norm in $\Lhat.$

\begin{defn}
Let $a \in G$ be fixed and $\blanknorm :\Lhat \to \R$ be the function defined for all $f \in \Lhat$ as
\begin{equation*}
    \fnorm := \fnormdef.
\end{equation*}
\end{defn}
Notice that this function depends on $a$, which is why we will be using the superscript $a$. Next, we will prove $\blanknorm $ is, in fact, a norm in $\Lhat$.
% We prove the defined norm is actually a norm
\begin{prop}
$\blanknorm$ is a norm in $\Lhat$ for any $a \in G$.
\end{prop}
\begin{proof}
Let $f \in \Lhat$. Clearly, if $f=0$, then $\fnorm=0$. Conversely, if $\fnorm = 0$ we then have
\begin{equation*}
    |f(a)|=0 \quad \text{ and }\quad \supfvw = 0
\end{equation*}
which implies $f(a)=0$ and $f(v)=f(w)$ for all $v \sim w$. But it is clear that this implies that $f(v)=0$ for all $v \in G$. It is straightforward to show that $\|\lambda f \|_{\Lhat}^{a} =|\lambda|\fnorm$, and that, $\fplusgnorm \leq \fnorm + \gnorm$ for every $f, g \in \Lhat$. Hence $\blanknorm$ is a norm in $\Lhat$, for any vertex $a$. 
\end{proof}

Now, as it was mentioned before, the value of the norm depends on the chosen vertex $a$ of $G$. So it is of interest to check if for any two vertices $a,b$ in $G$ the norms $\blanknorm$ and $\| \cdot \|_{\Lhat}^{b} $ are equivalent in $\Lhat$.

\begin{prop}\label{norm equivalence}
The norms $\blanknorm$ and $\| \cdot \|_{\Lhat}^{b} $ are equivalent in $\Lhat$ for any $a,b \in G$.
\end{prop}

\begin{proof}
We want to show there exist $M,N >0$ such that for all $f \in \Lhat$,
\begin{equation*}
    N\fnorm\leq \| f \|_{\Lhat}^{b}\leq M\fnorm.
\end{equation*}
First we will show the first inequality. Let $f \in \Lhat$, observe that
\begin{equation*}
|f(a)|=|f(a)-f(a_1)+f(a_1)- \dots -f(a_{n-1})+f(a_{n-1})-f(b)+f(b)|
\end{equation*}
 where $a=a_0\sim a_1 \sim \dots \sim a_{n-1}  \sim a_n=b$ forms a path from $a$ to $b$ and $\dist(a,b)=n$. Now by the triangle inequality we can see
\begin{equation*}
\openup 2\jot
    \begin{split}
    |f(a)|&\leq|f(a)-f(a_1)|+|f(a_1)-f(a_2)|+\dots+|f(a_{n-1})-f(b)| +|f(b)|\\
    &\leq n\supfvw +|f(b)|.
    \end{split}
\end{equation*}
Then, by adding $\supfvw$ to both sides of the inequality we get
\begin{equation*}
\openup 2\jot
    \begin{split}
    \fnorm &\leq (n+1)\supfvw + |f(b)|\\
    &\leq (n+1)\supfvw + (n+1)|f(b)|\\
    &=(n+1)\| f \|_{\Lhat}^{b}.
    \end{split}
\end{equation*}
Thus, we can see that the following inequality holds for all $f \in \Lhat$
\begin{equation*}
    \frac{1}{n+1}\fnorm\leq \| f \|_{\Lhat}^{b}.
\end{equation*}
Now, since $a,b$ where chosen arbitrarily and $\dist(a,b)=\dist(b,a)$, we can also conclude that for all $f \in \Lhat$,
\begin{equation*}
    \| f \|_{\Lhat}^{b}\leq (n+1)\fnorm.
\end{equation*}
Therefore there exists $n\in \N$ such that for all $f \in \Lhat$,
\begin{equation}
    \frac{1}{n+1}\fnorm \leq \| f \|_{\Lhat}^{b}
    \leq (n+1)\fnorm. \label{eq norm eq}
\end{equation}
Hence, for any $a,b \in G$ the norms $\blanknorm, \| \cdot \|_{\Lhat}^{b}$ are equivalent in $\Lhat$.
\end{proof}

To see the previous inequalities are sharp, take $a,b$ two vertices in $G$ such that $\dist(a,b)=n$. Observe that for the function $f\in \Lhat$ defined by $f(v)=\dist(a,v)$, we have 
\begin{equation*}
    \fnorm = \dist(a,a)+\sup\limits_{v\sim w}|\dist(a,v)-\dist(a,w)| = 1.
\end{equation*}
Also notice that
\begin{equation*}
    \|f\|_{\Lhat}^{b} = \dist(a,b)+\sup\limits_{v\sim w}|\dist(a,v)-\dist(a,w)| = n+1.
\end{equation*}
Hence it follows that
\begin{equation*}
    \|f\|_{\Lhat}^{b} = (n+1)\fnorm.
\end{equation*}
On the other hand, to see the inequality on the left of \eqref{eq norm eq} is sharp, take $f(v)=\dist(b,v)$.

We have established that $\Lhat$ is a normed vector space, and that all the norms we defined are equivalent. It remains to show the
completeness of $(\Lhat,\blanknorm)$, which does not depend on the choice of the vertex $a$ since we have shown the equivalence of the norms.

\begin{prop}
The space $\Lhat$ is complete with the norm $\blanknorm$.
\end{prop}
\begin{proof}
The proof follows the same lines as in \cite{CoEa}, so we just sketch the main ideas.

Let $(f_k)$ be a Cauchy sequence on $\Lhat$. Choose a vertex $v$ and let $n=\dist(a,v)$. Then,
\[
|f_k(v)| \leq \| f_k \|^v_{\Lhat} \leq (n+1) \|f_k \|_{\Lhat}^a,
\]
and hence, for each vertex $v$, the sequence $(f_k(v))$ is Cauchy and hence converges to a value $f(v)$.

Let $v\sim w$. Given the pointwise convergence of $(f_{k}(v))$ we know that there exists $N \in\N$ large enough such that $|f(v)-f_{N}(v)| < \frac{1}{2}$ and $| f(w)-f_{N}(w)| < \frac{1}{2}$. But hence
\begin{align*}    
 |f(v)-f(w)| 
    &\leq  |f(v)-f_{N}(v)| +|f_{N}(v)-f_{N}(w)| + |f_{N}(w)-f(w)|\\
    &\leq|f(v)-f_{N}(v)| + \| f_N \|_{\Lhat}^{a} +| f_{N}(w)-f(w)|\\
    &<1+\| f_N \|_{\Lhat}^{a}.
\end{align*}

Since $(f_k)$ is a Cauchy sequence in $(\Lhat,\blanknorm)$, it follows that $\|f_k \|_{\Lhat}^a \leq M$ for some $M \in \R$. This implies that $|f(v)-f(w)| \leq 1 +M$, which in turn implies that $f\in \Lhat$.

The proof that $\|f_k -f \|_{\Lhat}^a \to 0$ as $k \to \infty$ follows the same lines as the corresponding proof in \cite{CoEa}, so we omit it.
\end{proof}

%%%%%%%%%%%%%%%%%%%

The next proposition will be useful in what follows.

\begin{prop}\label{Proposition omega L}
For every $v \in G$, we have
\[
\omega(v) := \supfL = \dist(a,v).
\]
\end{prop}
\begin{proof}
The proof is essentially the same as \cite[Lemma 3.1]{CoEa}, but we include it for the sake of completeness. Let $g(v) = \dist(a,v)$ for $v \in G$. Notice that $g \in \Lhat$, that $g(a)=0$ and that $\|g\|_{\Lhat}^{a}=1$. Hence it follows that $\dist(a,v)=g(v)\leq \omega (v)$.

On the other hand, let $f$ be any function in $\Lhat$ such that $f(a)=0$ and $\fnorm \leq 1$. The equality will follow if we show that $|f(v)|\leq \dist(a,v)$ for any $v \in G$. We argue by induction on $\dist(a,v)$. First notice  $|f(a)|=0\leq \dist(a,a)$. Now suppose that for some $n \in \N_0$, we have that if $v \in G$ and $\dist(a,v)\leq n$, then  $|f(v)| \leq \dist(a,v)$. Take $z_1\sim z_2$ vertices in $G$ with $\dist(a,z_1)=n$ and $\dist(a,z_2)=n+1$. Then
\[        
|f(z_2)| \leq |f(z_2)-f(z_1)|+|f(z_1)| \leq \supfvw + |f(z_1)| \leq \fnorm + \dist(a,z_1) \leq 1 + n = \dist(a,z_2),
    \]
    which finishes the induction and the proof.
    \end{proof}

The next proposition finds a sharp estimate for the modulus of a function $f \in \Lhat$, which depends on the vertex $v$ in which the function is being evaluated.

\begin{prop} \label{Proposition upper bounds for f}
        If $f \in \Lhat$, then for all $z \in G$
        \begin{equation*}
            |f(z)|\leq |f(a)|+\dist(a,z)\supfvw.
        \end{equation*}
 In particular, for $z\in G$ 
\[
|f(z)| \leq \max\{ 1, \dist(a,z) \} \fnorm.
\]
Hence, if $\fnorm \leq 1$, then $|f(v)|\leq \dist(a,v)$ for each $v \in G, v\neq a$.
\end{prop}
\begin{proof}
The proof of the first part follows the same argument as \cite{CoEa}, but we include the proof for completeness. Notice that if $f$ is a constant function the result follows easily. Now, suppose $f$ is a nonconstant function in $\Lhat$ and let $C:=\supfvw \neq 0$. Consider the function $g:G\to \C$ defined by
\[
g(z)=\frac{1}{C} (f(z)-f(a)) \quad \text{ for } z\in G.
\]
Observe that $g(a)=0$ and $\gnorm =1$. By Proposition ~\ref{Proposition omega L}, it follows that $|g(z)|\leq \dist(a,z)$, for all $z \in G$. Hence, for any $f \in \Lhat$ and for all $z \in G$ we have
\[
|f(z)| \leq |f(a)| + |f(z)-f(a)| = |f(a)|+|g(z)|\supfvw \leq |f(a)|+\dist(a,z)\supfvw,
\]
which shows the first part of the proposition.

Now, for every $z \in G$ we have
\[
|f(z)| \leq |f(a)|+\dist(a,z)\supfvw \leq \max\{ 1, \dist(a,z) \} \fnorm,
\]
which shows the second part of the proposition. The last part follows trivially since if $z\neq a$, then $\max\{1,\dist(f,z)\}=\dist(a,z)$.
\end{proof}

The estimate in Proposition \ref{Proposition upper bounds for f} is sharp: take the function $f(v)=\dist(a,v)$ for all $v \in G$. As an easy consequence of the above inequality, we obtain the following result.

\begin{corollary}\label{cor:pointwise}
If $(f_n)$ is a sequence in $\Lhat$ such that $f_n \to 0$ as $n \to \infty$, then $f_n(v) \to 0$ as $n \to \infty$ for each $v \in G$.
\end{corollary}

Lastly, we can put together the information above into the following theorem.

\begin{theorem}\label{Prop Lhat is a functional Banach space}
    $\Lhat$ is a functional Banach space.
\end{theorem}
\begin{proof}
We already showed that $\Lhat$ is a Banach space of complex-valued functions. Clearly, there is no point in $G$ in which all functions vanish, and the estimate in Proposition \ref{Proposition upper bounds for f} implies that the point evaluation functionals are bounded. Hence $\Lhat$ is a functional Banach space. 
\end{proof}

We also have the following result, which was not shown in \cite{CoEa}.

\begin{theorem} \label{L not separable}
    The Lipschitz space $\Lhat$ is not separable.
\end{theorem}
\begin{proof}
Consider the subset $\mathcal{A}$ of $\Lhat$ given by
$$\mathcal{A}=\{f:G\to \{0,1\} \mid f(a)=0\}.$$ 
Let $f,g \in \mathcal{A}$ such that $f\neq g$. It is easy to see that there must exist neighboring vertices $w_1$ and $w_2$ such that $f(w_1)=g(w_1)$ and $f(w_2)\neq g(w_2)$. Hence, $|(f-g)(w_1) - (f-g)(w_2)|=1$ and since $f(a)=g(a)=0$, it follows that
\[
\| f-g \|_{\Lhat}^{a}= \sup\limits_{v \sim w} |(f-g)(v) - (f-g)(w)|\geq 1.
\]

Now enumerate all the vertices of $G$, except for $a$, and form the sequence $(v_n)$. Given $f\in \mathcal{A}$, define the sequence $(f_n)$ as $f_n=f(v_n)\in \{0,1\}$ for each $n\in \N$. Notice that this gives a bijective relation between $\mathcal{A}$ and the set of all sequences with values in $\{0,1\}.$ Since this set of sequences is uncountable, it follows that $\mathcal{A}$ is uncountable.

Since we have found a uncountable subset of $\Lhat$ such that all its members are at least one unit apart, it follows that $\Lhat$ cannot be separable.
\end{proof}

\section{The little Lipschitz space of a graph}

We now define the little Lipschitz space.

\begin{defn}\label{def f in L0}
The little Lipschitz space $\LAcerohat$ is defined as the set of all functions $f$ in $\Lhat$ such that
\[
    \lim\limits_{\dist(a,v)\rightarrow \infty}\hspace{.2cm} \maxfNv=0
\]
where $N_v:=\{w \in G \mid w\sim v\}$. Notice that since $G$ is locally finite, $\maxfNv$ is well defined.
\end{defn}

It is straightforward to show that $\LAcerohat$ is a vector subspace of $\Lhat$. We will now show that even though the limit of the distance is taken with respect to a specific fixed vertex $a$, the set of functions that satisfy such property does not vary if the vertex $a$ is changed.

\begin{prop}\label{prop: la igual lb}
$\LAcerohat = \Lhat_{0}^{b}$ for all $a,b \in G$.
\end{prop}
\begin{proof}
Take $a,b$ arbitrary vertices of $G$. Let $f\in \LAcerohat$, then by definition
\[
\lim\limits_{\dist(a,v)\rightarrow \infty}\hspace{.2cm} \maxfNv=0.
\]
This implies that for each $\epsilon > 0$ there exists $N \in \N$ such that
\[
\text{ if }    \dist(a,v)>N \text{ then } | f(v) - f(w)| <\epsilon \quad \text { for all } w \in N_v.
\]
Now define $N_1=N+\dist(a,b)$. Notice that if $v$ satisfies $\dist(b,v)> N_1$, then this implies that $\dist(a,v)>N$ and hence $| f(v) - f(w)| <\epsilon$ for all $w \in N_v$.

Hence we have shown that 
\[
\lim\limits_{\dist(b,v)\rightarrow \infty}\hspace{.2cm} \maxfNv=0.
\]
Therefore $f \in \Lhat_{0}^{b}$ and thus $\LAcerohat \subseteq \Lhat_{0}^{b}$. Since $a,b$ were chosen arbitrarily then $\LAcerohat=\Lhat_{0}^{b}$ for all $a,b \in G$.
\end{proof}

From this point forward we will not include the letter $a$ as a superscript in $\LAcerohat$ since we have just proved that the set $\LAcerohat$ is the same regardless of the value of the fixed vertex $a$ of $G$.

It is important to notice that the functions in $\Lcerohat$ are not necessarily bounded, as the next example shows.

\begin{ex}
Take the function $f:G\to\C$ defined by $f(a)=0$ and
\begin{equation*}
    f(v)=\sum\limits_{k=1}^{\dist(a,v)} \frac{1}{k} \quad \text{ for all } v \in G\setminus{\{a\}}.
\end{equation*}
Notice that when $\dist(a,v)\rightarrow \infty$ then $f(v)\rightarrow \infty$, as the harmonic sum diverges. Now take $v \in G$, with $v \neq a$ and let $v \sim w \in G$. If $\dist(a,v)=\dist(a,w)$, then 
\[
\left|\sum\limits_{k=1}^{\dist(a,v)} \frac{1}{k}-\sum\limits_{k=1}^{\dist(a,w)} \frac{1}{k}\right|=0.
\]
If $\dist(a,v)\neq \dist(a,w)$, then we may assume that $\dist(a,v)=\dist(a,w)+1$ and thus
\[
\left|\sum\limits_{k=1}^{\dist(a,v)} \frac{1}{k}-\sum\limits_{k=1}^{\dist(a,w)} \frac{1}{k}\right| =\frac{1}{\dist(a,v)+1}.
\]
In either case, we obtain that
\[
\max_{w \in N_v} \left|\sum\limits_{k=1}^{\dist(a,v)} \frac{1}{k}-\sum\limits_{k=1}^{\dist(a,w)} \frac{1}{k}\right| \leq \frac{1}{\dist(a,v)+1},
\]
and hence $f\in \Lcerohat$ (and also in $\Lhat$) and $f$ is not bounded.
\end{ex}

We continue our investigation of the space $\Lcerohat$ by showing that their values cannot grow faster than their distance to the root. In this case, the proof differs from that of \cite[Lemma 3.4]{CoEa}, even though the statement is the same. 

\begin{prop}\label{fn in L0}
       If $f \in \Lcerohat$, then
        \begin{equation*}
            \lim\limits_{\dist(a,v)\rightarrow \infty} \frac{f(v)}{\dist(a,v)}=0.
        \end{equation*}
\end{prop}
  \begin{proof} 
    Let $f \in \Lcerohat$. We first assume that $f(a)=0$.
    %Notice that if $f$ is a constant function the result %is trivial, so we may assume that $f$ is not constant %and hence $\supfvw >0$. 
    Let $\epsilon >0$.
    % Choose $\epsilon >0$ such that 
    % \[
    % \epsilon < \supfvw.
    % \]
    Since $f \in \Lcerohat$, there exists $N \in \N$ such that if  $\dist(a,v)>N$, then we have $|f(v)-f(w)|<\epsilon$ for all $w \sim v$.
    
    Let us take $z \in G$ such that $\dist(a,z)> N$. Consider a shortest path from $a$ to $z$, denoted by $a=u_0\sim u_1\sim \dots \sim u_m = z$, where $\dist(a,z)=m$ and $\dist(a,u_N)= N$. From Proposition \ref{Proposition upper bounds for f}, and since $f(a)=0$, we have
    \begin{equation}\label{eqq1}
         |f(u_N)|\leq \dist(a,u_N)\supfvw = N \supfvw .
    \end{equation}
    On the other hand, by the triangle inequality we have
    \begin{equation}\label{eqq2}
        |f(z)| \leq |f(u_N)|+ \sum\limits_{k=N+1}^{m}|f(u_k)-f(u_{k-1})|.
    \end{equation}
    Hence from the inequalities \eqref{eqq1} and \eqref{eqq2}, and since $\dist(a,u_k)>N$ for $k=N+1, N+2, \dots, m$, we can see that
    \begin{equation*}
    \begin{split}
        |f(z)|
        &\leq N \supfvw + (m-N)\epsilon \\
        &= N (\supfvw - \epsilon) + \dist(a,z)\epsilon.
         \end{split}
    \end{equation*}
The previous equation implies that
    \begin{equation*}
        \frac{|f(z)|}{\dist(a,z)} \leq \frac{N}{\dist(a,z)}\left(\supfvw-\epsilon\right) + \epsilon.
    \end{equation*}
This implies that
\begin{equation*}
        \limsup\limits_{\dist(a,z)\rightarrow \infty}\frac{|f(z)|}{\dist(a,z)}\leq\limsup\limits_{\dist(a,z)\rightarrow \infty}\frac{N}{\dist(a,z)}\left(\supfvw-\epsilon\right) + \epsilon.
\end{equation*}
Since $N$ only depends on the choice of $\epsilon$, we can therefore conclude that
\begin{equation*}
    \limsup\limits_{\dist(a,z)\rightarrow \infty}\frac{|f(z)|}{\dist(a,z)}\leq \epsilon.
\end{equation*}
Since $\epsilon>0$ can be taken arbitrarily small, it follows that
    \begin{equation*}
        \lim\limits_{\dist(a,v)\rightarrow \infty}\frac{f(v)}{\dist(a,v)}=0.
    \end{equation*}
    For the general case, when $f(a)$ does not necessarily equal $0$, let us define $g(v)=f(v)-f(a)$ for all $v \in G$. Notice that $g(a)=0$ and since 
    \begin{equation*}
        \maxgNv = \maxfNv
    \end{equation*}
    for any $v \in G$, then $g\in \Lcerohat$ if $f\in \Lcerohat$. Hence we can conclude
    \begin{equation*}
        \lim\limits_{\dist(a,v)\rightarrow \infty}\frac{f(v)}{\dist(a,v)}=\lim\limits_{\dist(a,v)\rightarrow \infty}\frac{f(a)}{\dist(a,v)}+\lim\limits_{\dist(a,v)\rightarrow \infty}\frac{g(v)}{\dist(a,v)}=0. \qedhere
    \end{equation*}
\end{proof}

We now prove that $\Lcerohat$ is a Banach space: this will be done by showing that $\Lcerohat$ is the closure of the set of functions with finite support. We set
\[
X=\{ f : G \to \C \mid f \text{ is of finite support} \}.
\]
Clearly, $X \subseteq \Lcerohat$. We now show that the closure of $X$ is in $\Lcerohat$.

Take $f$ in the closure of $X$ and fix $\epsilon>0$; we can then find $g_\epsilon \in X$ such that $\|f-g_\epsilon\|_{\Lhat}^{a}<\epsilon$. Hence,
\[
    |f(v)-g_\epsilon(v)-f(w)+g_\epsilon(w)|<\epsilon,
    \]
    for all $v \sim w \in G$.
    
Since $g_\epsilon \in X$, then there exists $M \in \N$ such that if $\dist(a,v)>M$ it follows that $g_\epsilon(v)=0$. Thus, it holds that
\[
|f(v)-f(w)|<\epsilon 
\]
for all $v\sim w$ such that $\dist(a,v)>M$ and $\dist(a,w)>M$. This implies that
\begin{equation*}
    \lim\limits_{\dist(a,v)\rightarrow \infty}\hspace{.2cm} \maxfNv=0
\end{equation*}
and thus $f\in \Lcerohat$.

We now show that, in fact, the closure of $X$ is $\Lcerohat$. Observe that this also differs from \cite{CoLo} and from \cite{CoEa}, where they show that the set of characteristic functions of ``sectors'' is dense.

% The set of functions with finite support is dense in L_0
\begin{theorem} \label{dense L0}
The closure of $X$  in $\Lhat$ equals $\Lcerohat$. In particular $\Lcerohat$ is a closed subspace of $\Lhat$ and hence is a separable Banach space.
\end{theorem}
\begin{proof}
Let $g \in \Lcerohat$ and $\epsilon >0$. We will show there exists a function $f \in X$ such that $\| g-f \|_{\Lhat}^{a} < \epsilon$. Since $g \in \Lcerohat$, and by Proposition~\ref{fn in L0}, there exists $N \in \N$ such that if $\dist(a,v)\geq N$ then $\maxgNv < \frac{\epsilon}{4}$ and $\frac{|g(v)|}{\dist(a,v)} < \frac{\epsilon}{4}$.

Define $f:G \to \C$ as
\begin{center}
    $ f(v) =  \begin{cases} \quad g(v), &  \quad  \text{ if }\quad   \dist(a,v) \leq N, \\   \quad  \frac{2N-\dist(a,v)}{\dist(a,v)}g(v),  & \quad \text{ if }\quad  N \leq \dist(a,v) \leq 2N, \\ \quad  0, & \quad \text{ if }\quad \dist(a,v) \geq 2N. \end{cases}$
\end{center}

Notice $f$ has finite support. Now we must only show that $\| g-f \|_{\Lhat}^{a} < \epsilon$. Since $f(a)=g(a)$, we need only show that
\begin{equation*}
    | g(v)-g(w)-f(v)+f(w)| < \epsilon\quad \text{ for all }  v\sim w.
\end{equation*}

We proceed by cases:

\begin{itemize}

\item Suppose $v\sim w$ are such that $\dist(a,v)\leq N$ and $\dist(a,w)\leq N$. Then
\begin{equation*}
|g(v)-g(w)-f(v)+f(w)|=|g(v)-g(w)-g(v)+g(w)| =0 < \epsilon.
\end{equation*}

% \item Suppose $v\sim w$ such that $\dist(a,v)=N$ and $\dist(a,w)= N+1$. Then
% \begin{align*}
% |g(v)-g(w)-f(v)+f(w)|&=\left| g(v)-g(w)-g(v)+\frac{2N-\dist(a,w)}{\dist(a,w)}g(w)\right|\\
%         &=\left|\frac{2N-(N+1)}{N+1}g(w)-g(w)\right|\\
%         &=\left|\frac{2}{N+1} \right||g(w)|\\
%         &=\frac{2|g(w)|}{\dist(a,w)}.
% \end{align*}
% Since $\dist(a,w)>N$, we have that $\frac{|g(w)|}{\dist(a,w)}<\frac{\epsilon}{4}$, which implies that
% \begin{equation*}
%     |g(v)-g(w)-f(v)+f(w)|<\epsilon.
% \end{equation*}

\item Suppose $v\sim w$ such that $N\leq\dist(a,v)\leq 2N$ and $N\leq\dist(a,w)\leq 2N$. Then
\begin{equation}\label{eq:N_2N_1}
\begin{split}
|g(v)-g(w)-f(v)+f(w)|&= \left|g(v)-\frac{2N-\dist(a,v)}{\dist(a,v)}g(v)-g(w)+\frac{2N-\dist(a,w)}{\dist(a,w)}g(w)\right|\\
&\leq |g(v)-g(w)|+\left|\frac{2N-\dist(a,v)}{\dist(a,v)}g(v)-\frac{2N-\dist(a,w)}{\dist(a,w)}g(w)\right|.
\end{split}
\end{equation}
Let us first estimate the second term on the right hand-side of inequality \eqref{eq:N_2N_1}. By the triangle inequality we obtain
\begin{equation}\label{eq:N_2N_2}
\begin{split}
\left|\frac{2N-\dist(a,v)}{\dist(a,v)}g(v)-\frac{2N-\dist(a,w)}{\dist(a,w)}g(w)\right|\leq {}& 
        \left|\frac{2N-\dist(a,v)}{\dist(a,v)}g(v)-\frac{2N-\dist(a,w)}{\dist(a,w)}g(v)\right| \\
        & \quad +\left|\frac{2N-\dist(a,w)}{\dist(a,w)}g(v)-\frac{2N-\dist(a,w)}{\dist(a,w)}g(w)\right|  \\
        = {} & \left|\frac{2N}{\dist(a,v)}
-\frac{2N}{\dist(a,w)} \right||g(v)| \\
& \quad + \left|\frac{2N-\dist(a,w)}{\dist(a,w)}\right||g(v)-g(w)|.
\end{split}
\end{equation}

Observe that since $v\sim w$ then $|\dist(a,v)-\dist(a,w)|\leq 1$ and since $\dist(a,w)\geq N$ then $\frac{N}{\dist(a,w)}\leq 1$; this implies that
\begin{equation}\label{eq:N_2N_3}
    \left|\frac{2N}{\dist(a,v)}-\frac{2N}{\dist(a,w)} \right||g(v)|\leq\left|\frac{2N}{\dist(a,v)\dist(a,w)}
\right||g(v)|\leq \frac{2|g(v)|}{\dist(a,v)}.
\end{equation}
On the other hand, since $N\leq\dist(a,w)\leq 2N$ then $0\leq \frac{2N-\dist(a,w)}{\dist(a,w)}\leq 1$ which implies
\begin{equation}\label{eq:N_2N_4}
    \left|\frac{2N-\dist(a,w)}{\dist(a,w)}\right||g(v)-g(w)|\leq|g(v)-g(w)|.
\end{equation}
Putting together inequalities \eqref{eq:N_2N_1}, \eqref{eq:N_2N_2}, \eqref{eq:N_2N_3} and \eqref{eq:N_2N_4} we obtain that
\begin{equation*}
    |g(v)-g(w)-f(v)+f(w)|\leq |g(v)-g(w)|+\frac{2|g(v)|}{\dist(a,v)}+|g(v)-g(w)|.
\end{equation*}
Since $\dist(a,v)\geq N$, then by the choice of $N$ we can conclude that
\begin{equation*}
    |g(v)-g(w)-f(v)+f(w)|<\frac{\epsilon}{4}+\frac{2\epsilon}{4}+\frac{\epsilon}{4}=\epsilon.
\end{equation*}

 % \item Suppose $v\sim w$ such that  $\dist(a,v)=2N-1$ and $\dist(a,w)=2N$. Then
 % \begin{equation*}
 %     \begin{split}
 %     |g(v)-g(w)-f(v)+f(w)| &= \left| g(v)-g(w)-\frac{2N-\dist(a,v)}{\dist(a,v)}g(v)+0\right|\\
 %     &=|g(v)-g(w)| + \frac{1}{\dist(a,v)} |g(v)| \\
 %     &< \frac{\epsilon}{4}+\frac{\epsilon}{4}\\
 %     &< \epsilon,
 %     \end{split}
 % \end{equation*}
 % since $\dist(a,v)>N$ and $\dist(a,w)>N$.

\item Suppose $v\sim w \in G$ such that  $\dist(a,v)\geq 2N$ and $\dist(a,w) \geq 2N$. Then $f(v)=f(w)=0$ which implies
\begin{equation*}
    |g(v)-g(w)-f(v)+f(w)|=|g(v)-g(w)|<\epsilon.
\end{equation*}
\end{itemize}

The cases above allow us to conclude that
\begin{equation*}
    | g(v)-g(w)-f(v)+f(w)| < \epsilon\quad \text{ for all }  v\sim w \in G,
\end{equation*}
and thus $\|g-f \|_{\Lhat}^{a} < \epsilon$ which proves $X$ is dense in $\Lcerohat$, as desired.

Finally, observe that since $X$ is clearly separable, it follows that $\Lcerohat$ is also a separable Banach space.
\end{proof}

The next proposition is the analogue of Proposition~\ref{Proposition omega L}, and will be useful in what follows.

\begin{prop}\label{Proposition omega L0}
For $v \in G$, \quad 
\[
\omega_0(v) := \supfLcero = \dist(a,v).
\]
\end{prop}
\begin{proof}
The proof follows the same lines as \cite[Lemma 3.1]{CoEa}. By Proposition \ref{Proposition omega L}, we have $\omega_0(v) \leq \omega(v) = \dist(a,v)$. Hence, it suffices to show $\dist(a,v)\leq \omega_0(v)$ for all $v \in G$.
Fix $v \in G$ and define the function
\[
f_v(w) = 
\begin{cases} 
\dist(a,w), & \text{ if } \dist(a,w) \leq \dist(a,v),\\
2 \dist(a,v)-\dist(a,w),  & \text{ if } \dist(a,v) \leq \dist(a,w) \leq 2 \dist(a,v),\\ 
0, & \text{ if } \dist(a,w) \geq 2 \dist(a,v).
\end{cases}
\]

Clearly  $f_v(a)=0$ and $f_v\in \Lcerohat$. It is straightforward to check that $\|f_v \|_{\Lhat}^{a}\leq 1$ and thus
\begin{equation*}
    f_v \in \{f \in \Lcerohat : \fnorm \leq 1 \text{ and } f(a)=0\}.
\end{equation*}
Since $f_v(v)= \dist(a,v)$ for all $v \in G\setminus \{a\}$, then 
\begin{equation*}
    \dist(a,v)=|f_v(v)|\leq \supfLcero=\omega_0(v),
\end{equation*}
which implies $\omega_0(v)= \dist(a,v)$ for all  $v \in G$.
\end{proof}

As in Theorem \ref{Prop Lhat is a functional Banach space}, we also obtain as a consequence of the above equality that $\Lcerohat$ is a functional Banach space.

\begin{theorem} \label{Prop Lcerohat is a functional Banach space}
    $\Lcerohat$ is a functional Banach space.
\end{theorem}

%%%%%%%%%%%%%%%%%%%%%%%%%%%%%%
%%%%%%%%%%%%%%%%%%%%%%%%%%%%%%
%%%%%%%%%%%%%%%%%%%%%%%%%%%%%%

\section{Weak and strong convergence of functions}

Later in this paper we will estimate the essential norm of multiplication operators. For this, we will need to consider weakly convergent sequences in $\Lcerohat$. The following result will be useful.

\begin{prop}\label{prop: weak convergence}
Let $(f_n)$ be a sequence in $\Lhat$ (respectively, in $\Lcerohat$). Assume that there exists $C>0$ such that for any subsequence $(f_{n_k})$ of $(f_n)$ and any sequence $(\theta_k)$ in $\C$ with $|\theta_k|=1$ for each $k\in \N$ we have
\[
\left\|\sum\limits_{k=1}^{M}\theta_k f_{n_k}\right\|_{\Lhat}^{a}\leq C \quad\text{ for all }M \in \N.
\]
Then $(f_n)$ converges to $0$ weakly in $\Lhat$ (respectively, in $\Lcerohat$).
\end{prop}
\begin{proof}
We will argue by contradiction. Suppose $(f_n)$ does not converge to $0$ weakly. That is, there exists a bounded linear functional $\ell \in (\Lhat)^{*}$ (respectively, in $(\Lcerohat)^{*}$) such that $\ell(f_n)\nrightarrow 0$. That is, there exists $\delta >0$ and a subsequence $(f_{n_k})$ of $(f_n)$ such that $|\ell(f_{n_k})|\geq \delta$ for all $k \in \N$.

For each $k \in \N$, we choose $\theta_k \in \C$ with $|\theta_k|=1$ such that $\theta_k \ell(f_{n_k})=|\ell(f_{n_k})|$. Now we define $h_k:=\theta_k f_{n_k}$ for each  $k \in \N$. Notice that for every $k \in \N$ we have
\begin{equation*}
    \ell(h_k)=\ell(\theta_k f_{n_k}) = \theta_k \ell(f_{n_k})=|\ell(f_{n_k})|.
\end{equation*}
This implies that $\ell(h_k)\geq \delta$ for all $k \in \N$. Hence, for all $M\in \N$ it follows that
\begin{equation*}
   \delta M \leq \sum\limits_{k=1}^{M}\ell(h_k) = \ell\left(\sum\limits_{k=1}^{M}h_k\right) =\ell\left(\sum\limits_{k=1}^{M}\theta_k f_{n_k}\right)\leq \|\ell\|\left\|\sum\limits_{k=1}^{M}\theta_k f_{n_k}\right\|_{\Lhat}^{a}. 
\end{equation*}

By hypothesis, there exists $C>0$ such that $\left\|\sum\limits_{k=1}^{M}\theta_k f_{n_k}\right\|_{\Lhat}^{a}\leq C$ for all $M\in \N$. Hence, it follows $\delta M \leq \|\ell\|C$ for all $M \in \N$. This is a contradiction since $\delta>0$ and thus we conclude the proof.
\end{proof}

We now have given sufficient conditions for a sequence $(f_n)$ in $\Lcerohat$ to be weakly convergent to $0$. The next result will characterize all the sequences $(f_n)$ in $\Lcerohat$ which converge strongly to $0$. To do so, we start by introducing the following definition:  

\begin{defn}\label{def: asint diminish}
Let $(f_n)$ be a sequence in $\Lcerohat$. We will say $(f_n)$ is \textit{asymptotically equidiminishing} in $\Lcerohat$ if for all $\epsilon>0$ there exists $N_\epsilon \in \N$ such that if $v\sim w$ and $\dist(a,v)>N_\epsilon$ then $|f_n(v)-f_n(w)|<\epsilon$, for all $n \in \N$.
\end{defn}

Now, we are able to characterize all the sequences $(f_n)$ in $\Lcerohat$ which converge strongly to $0$.

\begin{theorem}\label{asymptoticly equicontinuous}
Let $(f_n)$ be a sequence of functions in $\Lcerohat$. Then $(f_n)$ converges strongly to $0$ if and only if $(f_n)$ is asymptotically equidiminishing and $(f_n)$ converges to 0 pointwise.
\end{theorem}
\begin{proof}
First, suppose $(f_n)$ is asymptotically equidiminishing and $(f_n)$ converges to 0 pointwise.

Let $\epsilon>0$. Since $(f_n)$ is asymptotically equidiminishing, there exists $N_1$ such that if $v \sim w$ and $\dist(a,v) >N_1$, then $|f_n(v)-f_n(w)|< \frac{\epsilon}{2}$ for all $n\in \N$.

Since $(f_n)$ converges to $0$ pointwise, there exists $N_2$ such that $|f_n(v)| < \frac{\epsilon}{4}$ for all $n \geq N_2$ and for any vertex $v$ with $\dist(v,a)\leq N_1+1$ (observe that there are only finitely many such vertices).

Hence, for every vertex $v$ with $\dist(v,a)\leq N_1$ and $v \sim w$ we have
\[
|f_n(v)-f_n(w)| \leq |f_n(v)|+|f_n(w)| < \frac{\epsilon}{4} + \frac{\epsilon}{4} = \frac{\epsilon}{2}, \quad \text{ if } n \geq N_2.
\]
Hence, if $n \geq N_2$ then
\[
|f_n(v)-f_n(w)| < \frac{\epsilon}{2}, \quad \text{ for all } v \sim w.
\]
Therefore, for $n \geq N_2$,
\[
\sup_{v\sim w} |f_{n}(v)-f_{n}(w)| \leq \frac{\epsilon}{2}
\]
Therefore if  $n \geq N_2$, we have that
\begin{equation*}
    \left\|f_{n}\right\|_{\Lhat}^{a}=|f_n(a)|+\sup\limits_{v\sim w}|f_{n}(v)-f_{n}(w)| < \frac{\epsilon}{4}+\frac{\epsilon}{2}<\epsilon.
\end{equation*}
Thus, we have proved $(f_n)$ converges strongly to $0$. 

%%%%%%%%%%%%%

Now suppose $(f_n)$ converges strongly to $0$. That $(f_n)$ converges to zero pointwise follows from Corollary \ref{cor:pointwise}. Next we show that $(f_n)$ is asymptotically equidiminishing. Fix $\epsilon >0$. Since $\|f_n\|_{\Lhat}^{a} \to 0$, then there exists $N \in \N$ such that if $n\geq N$ if follows that 
\[
|f_n(a)|+\sup\limits_{v\sim w}|f_n(v)-f_n(w)|<\epsilon.
\]
Hence, if $n\geq N$ we have that $\sup\limits_{v\sim w}|f_n(v)-f_n(w)|<\epsilon$, from which it follows that $|f_n(v)-f_n(w)|<\epsilon$ for all $n\geq N$ and for all $v\sim w$.

On the other hand, notice that since for each $n < N$ the function $f_n$ is in $\Lcerohat$, we can choose $M \in \N$ such that if $\dist(a,v)>M $ and $v \sim w$, then $|f_n(v)-f_n(w)|<\epsilon$ for $n=1,2,\dots,N-1$. 

Hence we have found $M\in \N$ such that if $\dist(a,v)> M $ and $v\sim w$ then it follows that $|f_n(v)-f_n(w)|<\epsilon$ for all $n \in \N$. Therefore $(f_n)$ is asymptotically equidiminishing. 
\end{proof}

\section{Boundedness of the multiplication operators} \label{chap: Boundedness of the multiplication operators}

As usual, if $f:G \to \C$ we write $\|f\|_{\infty}:=\sup\limits_{v\in G}|f(v)|$, and we set $\hat{L}^{\infty}$ as the set
\[
\hat{L}^{\infty}:= \{f:G\to\C \, | \, \|f\|_{\infty}<\infty \}.
\]
The following definition will be useful in what follows.

\begin{defn} \label{Definition sigma psi}
For a function $\psi$ on $G$, we define $\ds \sigma_\psi := \sigmapsidef$.
\end{defn}

Observe that if $\sigma_\psi$ is finite then $\psi \in \Lcerohat$. Indeed, suppose $\sigma_\psi$ is finite, then $\dist(a,v)|\psi(v)-\psi(w)| < \sigma_{\psi} $ for any $v\sim w \in G$. Hence
\[
|\psi(v)-\psi(w)|<\frac{\sigma_\psi}{\dist(a,v)}\quad \text{ for all }v\sim w \in G,
\]
therefore
\[
    \lim\limits_{\dist(a,v)\rightarrow \infty} \sup\limits_{w\in N_v}|\psi(v)-\psi(w)|= 0
\]
and thus $\psi \in \Lcerohat$. 

On the other hand, there are functions in $\Lcerohat$ for which $\sigma_\psi$ is not finite. Indeed, it is not hard to see that if we set $\psi(v):=\sqrt{\dist(a,v)}$, then $\psi \in \Lcerohat$ but $\sigma_\psi =\infty$.

Recall that for $X$ a functional Banach space and $\psi$ a complex-valued function, a multiplication operator $M_\psi: X\to X$ is defined as $M_\psi(f)~=~\psi f$ for all $f \in X$. Observe that in order for $M_\psi$ to be well-defined, we need that $\psi f \in X$ for every $f\in X$. It is well known that this is enough for $M_\psi$ to be bounded.

% Important Lemma
\begin{lemma}{\cite[Lemma 11]{DuRoSh}} \label{Lemma psi f in lhat}
    Let $X$ be a functional Banach space on the set $\Omega$ and let $\psi$ be a complex-valued function on $\Omega$ such that $M_\psi$ maps $X$ into itself. Then $M_\psi$ is bounded on X and $|\psi(\omega)|\leq \|M_\psi\|$ for all $\omega \in \Omega$. In particular, $\psi$ is bounded.
\end{lemma}

Now, we can show the boundedness of the multiplication operator $M_\psi$ on $\Lhat$ is equivalent to the boundedness of $M_\psi$ on $\Lcerohat$ and also give a characterization in terms of the quantity $\sigma_\psi$. The proof follows the same lines as in \cite{CoEa}, but we include it here since there are a few details that differ, and also some parts are simplified.

\begin{theorem}\label{3 equivalence theorem}
    Let $G$ be a graph and $\psi$ a function on $G$. Then the following are equivalent statements:
    \begin{enumerate}
        \item $M_\psi$ is bounded on $\Lhat$. \label{uno}
        \item $M_\psi$ is bounded on $\Lcerohat$. \label{dos}
        \item $\psi \in \hat{L}^{\infty}$ and $\sigma_\psi$ is finite. \label{tres}
    \end{enumerate}
\end{theorem}
\begin{proof}
    (\ref{tres}) $\implies$ (\ref{uno}). Assume $\psi \in \hat{L}^{\infty}$ and $\sigma_\psi$ is finite. We will show that $\psi f \in \Lhat$.
    Notice that for vertices $x\sim y$ we have
    \[
    |\psi(x)f(x)-\psi(y) f(y)| \leq |\psi(x)-\psi(y)| \, |f(x)|+ |\psi(y)|\, |f(x)-f(y)|.
    \]
    Now, by Proposition ~\ref{Proposition upper bounds for f}, and recalling that since $\sigma_\psi$ is finite, then $\psi \in \Lhat$, we have
    \begin{equation*}
    \begin{split}
    |\psi(x)f(x)-\psi(y) f(y)| 
    \leq & |\psi(x)-\psi(y)| \, \left( |f(a)|+\dist(a,x)\supfvw \right) + |\psi(y)| \supfvw \\
    \leq & |f(a)| \, \sup\limits_{v\sim w}|\psi(v)-\psi(w)|    +  \sigma_\psi \supfvw + \|\psi\|_\infty \supfvw,
    \end{split}
    \end{equation*}
    which implies that
    \begin{equation}\label{ineq_psi_f}
    \suppsifvw \leq |f(a)| \, \sup\limits_{v\sim w}|\psi(v)-\psi(w)| + (\sigma_\psi + \|\psi\|_\infty) \supfvw,
    \end{equation}
     Therefore, $\psi f \in \Lhat$. Finally, by Lemma ~\ref{Lemma psi f in lhat} we can conclude that $M_\psi$ is bounded on~$\Lhat$.
    
    (\ref{uno}) $\implies$ (\ref{tres}). Suppose $M_\psi$ is bounded on $\Lhat$. In particular, this means that $M_\psi$ maps $\Lhat$ into itself. Then by Lemma ~\ref{Lemma psi f in lhat}, we know that $\|\psi\|_{\infty} \leq \|M_\psi\|$ and hence $\psi \in \hat{L}^{\infty}$. It just remains to be seen that $\sigma_\psi$ is finite.
    
    Let $f(z)=\dist(a,z)$ for each $z \in G$. It is clear that $f \in \Lhat$ and $\fnorm=1$. Then for each $x\sim y$ we have
\begin{align*}
            |f(x)| \ |\psi(x)-\psi(y)|
            &=|f(x)\psi(x)-f(x)\psi(y)|\\
            &\leq |f(x)\psi(x)-f(y)\psi(y)|+|\psi(y)|\ |f(y)-f(x)|\\
            &\leq |f(x)\psi(x)-f(y)\psi(y)|+\|\psi\|_{\infty} \ |f(x)-f(y)|\\
            &\leq\|M_{\psi}f\|_{\Lhat}^{a}+\|\psi\|_{\infty}\ \supfvw\\
            &\leq \|M_{\psi}\|+\|\psi\|_{\infty}.
\end{align*}
Hence
    \begin{equation*}
        \dist(a,x)|\psi(x)-\psi(y)|\leq \|M_{\psi}\|+\|\psi\|_{\infty}\quad \text{ for all }x \sim y \in G.
    \end{equation*}
    Therefore, $\sigma_\psi$ is finite.

    (\ref{tres}) $\implies$ (\ref{dos}). Assume $\psi \in \hat{L}^{\infty}$ and $\sigma_\psi$ to be finite. For $f \in \Lcerohat$ and $x \sim y \in G$ with $x\neq a$ we have
    \begin{equation*}
        \begin{split}
    |\psi(x)f(x)-\psi(y)f(y)|
    &\leq|\psi(x)-\psi(y)| \ |f(x)|+|\psi(y)| \ |f(x)-f(y)|\\
    &\leq \frac{|f(x)|}{\dist(a,x)}\dist(a,x)|\psi(x)-\psi(y)| +\|\psi\|_{\infty} \ |f(x)-f(y)|\\
       &\leq \frac{|f(x)|}{\dist(a,x)} \sigma_\psi+\|\psi\|_{\infty} \max_{z \in N_x}\{|f(x)-f(z)|\}
        \end{split}
    \end{equation*}
Since $\psi \in \hat{L}^{\infty}$ and $\sigma_\psi$ is finite, we have by Proposition \ref{fn in L0} and Definition \ref{def f in L0} respectively, that 
    \begin{equation*}
        \frac{|f(x)|}{\dist(a,x)}\sigma_\psi \to 0 \quad \text{ and } \quad     
         \|\psi\|_{\infty}\max_{z \in N_x}\{|f(x)-f(z)|\} \to 0 \quad \text{ as } \dist(a,x) \to \infty.
    \end{equation*}
    Hence
    \begin{equation*}
|\psi(x)f(x)-\psi(y)f(y)| \leq \frac{|f(x)|}{\dist(a,x)} \sigma_\psi+\|\psi\|_{\infty} \max_{z \in N_x}\{|f(x)-f(z)|\}\to 0 \quad \text{ as } \dist(a,x) \to \infty
    \end{equation*}
    which implies that $\psi f \in \Lcerohat$. Since we have shown that (\ref{tres}) $\implies$ (\ref{uno}), we know that $M_\psi$ is bounded on $\Lhat$, and thus the boundedness of $M_\psi$ on $\Lcerohat$ follows from $M_\psi(\Lcerohat) \subseteq \Lcerohat$.
    
    (\ref{dos}) $\implies$ (\ref{tres}). First of all, since $\Lcerohat$ is a functional Banach space, and $M_\psi$ is bounded on $\Lcerohat$, Lemma \ref{Lemma psi f in lhat} implies that $\psi \in \hat{L}^{\infty}$. Secondly, as we did in the proof of (\ref{uno}) $\implies$ (\ref{tres}), for any function $f\in \Lcerohat$ with $f(a)=0$ and $\fnorm \leq 1$, and for every $x \sim y$ we have
\[
|f(x)| \, |\psi(x)-\psi(y)| \leq \|M_{\psi}\|+\|\psi\|_{\infty}.
\]
If we take the supremum over all such functions $f$, by Proposition ~\ref{Proposition omega L0}, we have
  \[
\dist(a,x) |\psi(x)-\psi(y)| \leq \|M_{\psi}\|+\|\psi\|_{\infty}
\]  
and hence $\sigma_\psi \leq \|M_{\psi}\| + \|\psi\|_{\infty}$.
\end{proof}

As was shown in \cite{CoEa}, observe that there are functions $\psi$ not in $\hat{L}^{\infty}$ for which $\sigma_\psi$ is finite, and functions $\psi$ in $\hat{L}^{\infty}$ for which $\sigma_\psi$ is infinite.

%%%%%%%%%%%%%%%%%%%%%%%%%%%%%%%%

We now provide some estimates for the norm of the multiplication operator on $\Lhat$ and $\Lcerohat$.

\begin{theorem}\label{prop bounds for the operator}
    If $M_\psi$ is a bounded multiplication operator on $\Lhat$ (or equivalently, on $\Lcerohat$) then
    \[
    \max\{ \psinorm, \psisup\} \leq \|M_\psi\|\leq \psisup + \sigma_\psi.
    \]
\end{theorem}
\begin{proof}
The proof for the lower estimate is identical to \cite[Theorem 4.1]{CoEa}, so we omit it. To show that $\|M_\psi\|\leq \psisup + \sigma_\psi$, notice that for $f \in \Lhat$, inequality \eqref{ineq_psi_f} in Theorem~\ref{3 equivalence theorem} implies that
\begin{equation*}
    \begin{split}
    \|M_\psi f\|_{\Lhat}^{a}&=|\psi(a)f(a)|+\suppsifvw\\
        &\leq |\psi(a)| \, |f(a)| +|f(a)|\sup\limits_{v\sim w}|\psi(v)-\psi(w)|+(\sigma_\psi + \|\psi\|_{\infty})\supfvw.
    \end{split}
\end{equation*}
Notice that $|\psi(a)|\leq \psisup$ and $\sup\limits_{v\sim w}|\psi(v)-\psi(w)|\leq \sigma_\psi$ which implies that 
\begin{equation*}
    \|M_\psi f\|_{\Lhat}^{a} \leq|f(a)|(\sigma_\psi+\psisup)+(\sigma_\psi + \|\psi\|_{\infty})\supfvw
    =(\sigma_\psi + \|\psi\|_{\infty})\, \fnorm,
\end{equation*}
which shows that $\|M_\psi\|\leq \sigma_\psi + \|\psi\|_{\infty}$. The same argument applies to $f \in \Lcerohat$.
\end{proof}

The estimates above are sharp. Take $\psi$ to be a constant function to see that the lower inequality holds, and to see that the upper estimate is sharp, take $\psi$ to be the characteristic function of $a$. See \cite{CoEa} for details.

%%%%%%%%%%%%%%%%%%%%%%%%%%%%%%%%%%%%%%

\section{The spectrum of a bounded multiplication operator}

The proof of the following theorem follows the same lines as \cite{CoEa}. We offer a sketch of the proof for completeness.

\begin{theorem} \label{spectrum proposition}
    Let $M_\psi$ be a bounded multiplication operator on $\Lhat$ or $\Lcerohat$ then:
    \begin{itemize}
    \item $\sigma_{\rm p}(M_\psi) = \psi (G)$.
        \item $\sigma(M_\psi) = \sigma_{\rm ap}(M_\psi) = \overline{\psi(G)}$.
    \end{itemize}
\end{theorem}
\begin{proof}
By Theorem \ref{3 equivalence theorem}, it is enough to show the result when $M_{\psi}$ is a bounded operator on $\Lhat$.

The proof that $\sigma_{\rm p}(M_\psi) = \psi (G)$ is the same as the one in \cite{CoEa} so we omit it. Next, we will show that $\sigma(M_\psi)= \overline{\psi(G)}$. Since we know that $\psi(G)=\sigma_{\rm p}(M_\psi)$ we have $\overline{\psi(G)} \subseteq \sigma(M_\psi)$. Conversely, to show that $\sigma(M_\psi) \subseteq \overline{\psi(G)}$, consider $\lambda \in \C$ such that $\lambda \not\in \overline{\psi(G)}$. Then there exists $c>0$ such that $|\psi(v)-\lambda|\geq c$ for all $ v \in G.$ Thus the function $\varphi_\lambda$ defined by $\varphi_\lambda (v) = \frac{1}{\psi(v)-\lambda}$ is bounded on $G$ by $\frac{1}{c}$. From this it follows that
\[
\sigma_{\varphi_{\lambda}} \leq \frac{1}{c^{2}} \sigma_\psi.
\]
Since $M_\psi$ is bounded on $\Lhat$ then by Theorem \ref{3 equivalence theorem} we know that $\sigma_\psi$ is finite which in turns implies that $\sigma_{\varphi_{\lambda}}$ is also finite. Hence, we now have that $\sigma_{\varphi_{\lambda}}$ is finite and $\varphi_{\lambda} \in \hat{L}^{\infty}$ since $\varphi_\lambda$ is bounded by $\frac{1}{c}$. Thus, by Theorem \ref{3 equivalence theorem} the operator $M_{\varphi_{\lambda}}$ is bounded on $\Lhat$. One easily checks that $(M_{\psi}-\lambda)^{-1}=
M_{\varphi_{\lambda}}$, finishing the proof that $\sigma(M_\psi) = \overline{\psi(G)}$.  Lastly, since $\sigma_{\rm ap}(M_\psi)$ is a closed set, the theorem follows.
\end{proof}

\section{Compactness of the multiplication operators}
\label{chap: Compactness of the multiplication operators}

In this section we prove a multiplication operator is compact in $\Lhat$ if and only if it is compact in $\Lcerohat$. Furthermore, we give a complete characterization for compact multiplication operators acting on these spaces.  

\begin{prop}
A bounded multiplication operator $M_\psi$ on $\Lhat$ (respectively on $\Lcerohat$) is compact if and only if for every bounded sequence $(f_n)$ in $\Lhat$ (respectively in $\Lcerohat$) that converges to $0$ pointwise, the sequence $(\psi f_n)$ converges to 0 in norm. \label{characterization bounded operators}
\end{prop}
\begin{proof}
We prove the result for bounded multiplication operators acting on $\Lhat$. The proof for $M_\psi$ acting on $\Lcerohat$ follows the same argument.

Assume that $M_\psi$ is compact on $\Lhat$ (the proof of this part is essentially the same as in \cite[Lemma 7.1]{CoEa}, and we include a sketch of it for completeness). Take $(f_n)$ a bounded sequence in $\Lhat$ converging to 0 pointwise. We may assume $\|f_n\|_{\Lhat}^{a} \leq 1$ for all $n \in \N$. Since $M_\psi$ is compact the sequence $(f_n)$ has a subsequence $(f_{n_k})$ such that $(\psi f_{n_k})$ converges in norm to some function $f \in \Lhat$. Now observe that since $\|\psi f_{n_k} - f\|_{\Lhat}^{a}\to 0$,  Proposition \ref{Proposition upper bounds for f} implies that $\psi f_{n_k} \to f$ pointwise. Given that $f_n \to 0$ pointwise then $\psi f_{n_k}\to 0$ pointwise, which implies that $f$ is the zero function and therefore $\|\psi f_{n_k}\|_{\Lhat}^{a}\to 0$.  The argument above shows that in fact $0$ is the unique limit point in $\Lhat$ for the sequence $(\psi f_n)$.  Since $M_\psi$ is a compact operator, and the sequence $(\psi f_n)$ has a unique limit point, it follows that $\|\psi f_{n}\|_{\Lhat}^{a}\to 0$, as desired.

Conversely, suppose that for every bounded sequence $(f_n)$ in $\Lhat$ converging to $0$ pointwise, the sequence $\|\psi f_n\|_{\Lhat}^{a}$ converges to $0$. We will show that $M_\psi$ is compact: to do this, let $(g_n)$ be a sequence in $\Lhat$ with $\|g_n\|_{\Lhat}^{a} \leq 1$ for all $n \in \N$. Then by Proposition \ref{Proposition upper bounds for f} it follows that $|g_n(v)|\leq \dist(a,v)$ for each $v \in G$ and for all $n \in \N$; that is, $(g_n)$ is bounded pointwise. From \cite[Theorem 1.4]{ReSi}, there exists a subsequence $(g_{n_k})$ of $(g_n)$ that converges pointwise to a function $g$. 

To see that $g \in \Lhat$, let $v \sim w \in G$ and observe that we can find $N_v,N_w \in \N$ such that if $k \geq N_v$ and $k \geq N_w$ then $|g(v)-g_{n_k}(v)|<1$ and $|g_{n_k}(w)-g(w)|<1$. Let $M=\max\{N_v,N_w\}$. Then
\begin{align*}
    |g(v)-g(w)|&\leq |g(v)-g_{n_M}(v)|+|g_{n_M}(w)-g(w)|+|g_{n_M}(v)-g_{n_M}(w)|\\
    &\leq |g(v)-g_{n_M}(v)|+|g_{n_M}(w)-g(w)| + \|g_n\|_{\Lhat}^{a}\\
    &< 3.
\end{align*}
This implies that $\supgvw<\infty$ and thus $g \in \Lhat$. Furthermore, since the sequence $(g_n)$ is bounded in $\Lhat$ we can now see that the sequence $(f_k)$ defined as $f_k =g_{n_k}-g$ for all $k \in \N$ is also bounded in $\Lhat$. Moreover, we know that the sequence $(f_k)$ converges to $0$ pointwise.  Hence, by hypothesis we have that $\|\psi f_k\|_{\Lhat}^{a} \to  0$ as $n \to \infty$; that is, $\psi g_{n_k}$ converges to $\psi g$ in norm. Since this proves that $M_\psi$ transforms bounded sequences to sequences which have a convergent subsequence, the operator $M_\psi$ is compact.
\end{proof}

One can show from the above proposition that if $\psi$ has finite support then $M_\psi$ is a compact operator. On the other hand, we can also show from the previous proposition that if $\psi$ is a constant function different to the zero function then $M_\psi$ is not a compact operator. This will follow more easily from Theorem \ref{3 equiv compact operator} below.

Now, similarly to Theorem \ref{3 equivalence theorem}, we can show the compactness of the multiplication operators $M_\psi$ on $\Lhat$ is equivalent to the compactness of $M_\psi$ on $\Lcerohat$ and also give a characterization in terms of two quantities which depend on $\psi$. Again, the proof follows the same lines as in \cite[Theorem 7.2]{CoEa}, but it differs in a few details.

\begin{theorem}\label{3 equiv compact operator}
    Let $M_\psi$ be a bounded multiplication operator on $\Lhat$ (or equivalently on $\Lcerohat$). Then the following three statements are equivalent:
    \begin{enumerate}
        \item $M_\psi$ is compact on $\Lhat$
        \item $M_\psi$ is compact on $\Lcerohat$
        \item The following conditions hold:
      \end{enumerate}
      \begin{equation*}
\lim\limits_{\dist(a,v)\to\infty}\psi(v)=0 \quad \text{ and } \quad \lim\limits_{\dist(a,v)\to\infty}\dist(a,v)\max\limits_{w \in N_{v}} | \psi (v)-\psi(w)|=0 
      \end{equation*}
\end{theorem}
\begin{proof}
We will only show the equivalence $(1) \iff (3)$, since the equivalence $(2) \iff (3)$ is similar.

$(1)\implies (3)$. Assume $M_\psi$ is compact on $\Lhat$. It suffices to show that given a sequence of vertices $(v_n)$ in $G$  with $v_n\neq a$ for all $n\in \N$ such that $\dist(a,v_n)\to \infty$ as $n\to \infty$, it follows that
\begin{equation}\label{compact_nec_suff}
\lim\limits_{n\to \infty} \psi(v_n)=0 \quad \text{ and } \quad \lim\limits_{n\to \infty} \dist(a,v_n)\max\limits_{w\in N_{v_n}} | \psi (v_n)-\psi(w)|=0.
\end{equation}
Let $(v_n)$ be such a sequence and for each $n \in \N$ let $f_n$ be characteristic function of the vertex $v_n$. It is clear that $f_n \to 0$ pointwise and $\|f_n\|_{\Lhat}=1$ for all $n \in \N$. By Proposition \ref{characterization bounded operators} it follows that $\|\psi f_n\|_{\Lhat}^{a} \to 0$ as $n \to \infty$. But notice that 
\[
\|\psi f_n\|_{\Lhat}^{a}=|\psi(a) f_n(a)|+ \sup\limits_{v\sim w}|\psi(v)f_n(v)-\psi(w)f_n(w)|=|\psi(v_n)|,
\]
hence it follows that $\lim\limits_{n\to \infty} \psi(v_n)=0$ and thus the first part of condition \eqref{compact_nec_suff} holds. Now, for each $n \in \N$, let us define the function $g_n : G \to \C$ as
\[g_n(v) =  \begin{cases} 
0, \quad &\text{ if } \quad   \dist(a,v) < \left\lfloor\frac{\dist(a,v_n)}{2} \right\rfloor, \\ 
2\dist(a,v)-\dist(a,v_n)+2, \quad &\text{ if }\quad \left\lfloor\frac{\dist(a,v_n)}{2} \right\rfloor  \leq \dist(a,v) <\dist(a,v_n), \\ 
\dist(a,v_n), \quad &\text{ if }\quad \dist(a,v) \geq \dist(a,v_n). \\ \end{cases}
\]
Observe that $g_n \to 0$ pointwise. Also, it is straightforward to check that $\|g_n\|_{\Lhat}^{a} \leq 2$ for all $n \in \N$. Therefore, by Proposition \ref{characterization bounded operators} we obtain that $\|  \psi g_n \|_{\Lhat}^{a} \to 0$. 

Let $n \in \N$ fixed. Observe that if $w\sim v_n$, and $\dist(a,w) \geq \dist(a,v_n)$, we have $g_n(w) = \dist(a,v_n) = g_n(v_n)$. On the other hand, if $w\sim v_n$, and $\dist(a,w) < \dist(a,v_n)$, then $\dist(a,w)= \dist(a,v_n)-1$ and hence
\[
    g_n(w) = 2(\dist(a,v_n)-1)-\dist(a,v_n)+2 = \dist(a,v_n) = g_n(v_n).
\]
Hence, if $w\sim v_n$ it follows that $g_n(w)=g_n(v_n)=\dist(a,v_n)$ and thus we can see that
\[
\|\psi g_n\|_{\Lhat}^{a} \geq \max\limits_{w\in N_{v_n}} |\psi(v_n)g_n(v_n) - \psi (w)g_n(w)|\\
= \dist(a,v_n)\max\limits_{w\in N_{v_n}} |\psi(v_n)- \psi (w)|
\]
for all $n\in \N$. Therefore, since $\|\psi g_n\|_{\Lhat}^{a} \to 0$ then it follows that $\lim\limits_{n\to \infty} \dist(a,v_n)\max\limits_{w\in N_{v_n}} | \psi (v_n)-\psi(w)|=0$ and thus both conditions in condition \eqref{compact_nec_suff} have been shown.

$(3)\implies (1)$ Assume that both conditions of $(3)$ hold. Since $(1)$ is trivially true if $\psi$ is identically zero, we may assume that $\psi$ is not the zero function.  By Proposition \ref{characterization bounded operators}, to prove that $M_\psi$ is compact it suffices to show that if $(f_n)$ is a sequence in $\Lhat$ converging to 0 pointwise and 
$0 \neq \sup\limits_{n \in \N}\|f_n\|_{\Lhat}^{a}<\infty$ then $\|\psi f_n\|_{\Lhat}^{a} \to 0 \text{ as } n \to \infty.$

Let $(f_n)$ be such a sequence and fix an $\epsilon>0$. Since $f_n \to 0$ pointwise and $\psi \in \Lhat$, there exists $N \in \N$ such that if $n \geq N$ then
\begin{equation}\label{seis uno}
|f_n(a)|<\frac{\epsilon}{3\psinorm}.
\end{equation}
By $(3)$ we know that there exists $M>0$ such that if $\dist(a,v)\geq M$ then
\begin{equation}\label{psi_small}
|\psi(v)|<\frac{\epsilon}{3\sup\limits_{n \in \N}\|f_n\|_{\Lhat}^{a}} \quad \text{ and } \quad 
\dist(a,v)\max\limits_{w \in N_{v}} |\psi (v)-\psi(w)|<\frac{\epsilon}{3\sup\limits_{n \in \N}\|f_n\|_{\Lhat}^{a}}.
\end{equation}
Now observe that for every $v\in G$,
\begin{equation*}
    \begin{split}
        \max\limits_{w \in N_{v}} | \psi (v)f_n(v)-\psi(w)f_n(w)|&\leq |f_n(v)| \max\limits_{w \in N_{v}} | \psi (v)-\psi(w)|\\ &\hspace{.6cm} +\max\limits_{w \in N_{v}}|\psi(w)| \max\limits_{w \in N_{v}} | f_n(v)-f_n(w)|\\ 
    \end{split}
\end{equation*}
and by Proposition \ref{Proposition upper bounds for f} it follows that
\begin{align*}
        \max\limits_{w \in N_{v}} | \psi (v)f_n(v)-\psi(w)f_n(w)| &\leq(|f_n(a)|+\dist(a,v)\sup\limits_{x \sim y} |f_n(x)-f_n(y)|)
        \max\limits_{w \in N_{v}} |\psi(v)-\psi(w)|\\
        &\hspace{1cm}+\max\limits_{w \in N_{v}}|\psi(w)| \max\limits_{w \in N_{v}} | f_n(v)-f_n(w)|\\
        &= |f_n(a)|\max\limits_{w \in N_{v}} |\psi(v)-\psi(w)|\\&\hspace{1cm}+\dist(a,v)\sup\limits_{x \sim y} |f_n(x)-f_n(y)|
        \max\limits_{w \in N_{v}} |\psi(v)-\psi(w)|\\
        &\hspace{1cm}+\max\limits_{w \in N_{v}}|\psi(w)| \max\limits_{w \in N_{v}} | f_n(v)-f_n(w)|\\
        &\leq |f_n(a)| \psinorm + \dist(a,v) \max\limits_{w \in N_{v}} |\psi(v)-\psi(w)| \, \|f_n\|_{\Lhat}^{a}+ \max\limits_{w \in N_{v}}|\psi(w)| \, \|f_n\|_{\Lhat}^{a}.
\end{align*}

By expressions \eqref{seis uno} and \eqref{psi_small}, it now follows that if we take $n\geq N$, then for all $v\sim w \in G$ such that $\dist(a,v)> M$ then 
\begin{equation}\label{la buena}
        \max\limits_{w \in N_{v}} | \psi (v)f_n(v)-\psi(w)f_n(w)|
        < \epsilon.
\end{equation}

On the other hand, since $f_n \to 0$ pointwise in $\Lhat$ then $f_n \to 0$ uniformly on the finite set $\{v \in G: \dist(a,v) \leq M\}$. Since $\psi$ is bounded, this implies, rechoosing $N$ if necessary, that if $n \ge N$ then
\[
\max\limits_{w \in N_v} |\psi(v) f_n(v)-\psi(w) f_n(w)| < \epsilon,
\]
for $v \in G$ with $\dist(a,v) \leq M$. Therefore, together with expression \eqref{la buena} we can conclude that for all $v\in G$ and for $n\geq N$ we have that
\begin{equation}\label{la buena 2}
    \max\limits_{w \in N_{v}} | \psi (v)f_n(v)-\psi(w)f_n(w)|<\epsilon.
\end{equation}

Furthermore, given that $f_n(a)\to 0$ pointwise and together with expression \eqref{la buena 2} we have that $\|\psi f_n\|_{\Lhat}^{a} \to 0$ as $n \to \infty$, which by Proposition \ref{characterization bounded operators} concludes the proof that $M_\psi$ is compact. 
\end{proof}

The conditions (3) above are independent. Indeed, observe that the function $\psi$ defined by 
\[
\psi(v)=\begin{cases}
    1,& \text{ if } v=a,\\
\frac{\sin(\dist(a,v))}{\dist(a,v)},& \text{ if } v \neq a,
\end{cases}
\]
satisfies the first condition, but not the second one. On the other hand, the function 
\[
\psi(v)=\sum_{k=1}^{\dist(a,v)+1} \frac{1}{k^2}
\]
satisfies the second condition, but not the first one. (Note that in both cases the multiplication operator defined by the functions is bounded.)

\section{Estimates on the essential norm of a multiplication operator}
\label{chap: Estimates for the essential norm of multiplication operators}

In this section we provide estimates for the essential norm of a bounded multiplication operator acting on $\Lhat$. Recall that the \textit{essential norm} of an operator $A$ on a Banach space $X$, denoted as $\|A\|_{e}$ is the distance in the operator norm from $A$ to the set of compact operators on $X$; that is, 
\begin{equation*}
    \|A\|_{e} = \inf\{\|A-K\| : K \text{ a compact operator on } X\}.
\end{equation*}

We will now define two quantities which will be useful to bound the essential norm of $M_\psi$.

\begin{defn}
Let $\psi$ be a complex-valued function on $G$ such that the multiplication operator $M_\psi$ is bounded on $\Lhat$. We define

\begin{align*} 
A(\psi)&=\lim\limits_{n\to \infty}\sup\left\{|\psi(v)| \,:\, \dist(a,v)\geq n\right\},\\
B(\psi)&= \lim\limits_{n\to \infty}\sup\left\{\dist(a,v)|\psi(v)-\psi(w)|\mid v\sim w, \dist(a,v), \dist(a,w)\geq n\right\}.\\
\end{align*}
\end{defn}
Since $M_\psi$ is bounded, we know by Theorem \ref{3 equivalence theorem} that $\psi~\in~\Lhat^{\infty}$ and $\sigma_\psi$ are finite and the sequences of supremums are monotone decreasing which implies that the limits in the above expressions exist and are finite.

In the next proposition we will prove $A(\psi)$ is a lower bound for the essential norm of $M_\psi$.

\begin{prop}\label{norma esencial abajo}
Let $M_\psi$ be bounded on $\Lhat$ or $\Lcerohat$. Then
\begin{equation*}
    \|M_\psi\|_{e} \geq A(\psi).
\end{equation*}
\end{prop}

\begin{proof}
Let us enumerate all the vertices of $G$ as $v_1,v_2,v_3,\dots$, with $v_1=a$, and consider the sequence $(v_n)$ in $G$. Let $\chi_{n}:G\to \{0,1\}$ denote the characteristic function for the vertex $v_n$. Clearly $\chi_n$ is in $\Lcerohat$ for each $n \in \N$.

We would like to show that the sequence $(\chi_n)$ converges to $0$ weakly. To do so, by Proposition \ref{prop: weak convergence}, it is enough to show  that for any $(\chi_{n_k})$ subsequence of $(\chi_n)$ and any sequence $(\theta_k)$ in $\C$ with $|\theta_k|=1$ for each $k\in \N$, we have
\[
\left\|\sum\limits_{k=1}^{M}\theta_k \chi_{n_k}\right\|_{\Lhat}^{a}\leq 3 \quad\text{ for all }M \in \N.
\]
Clearly, for all $M \in \N$ and for vertices $v\sim w$ we have that
\begin{equation*}
    \left|\sum\limits_{k=1}^{M}\theta_k\chi_{n_k}(v)-\sum\limits_{k=1}^{M}\theta_k\chi_{n_k}(w)\right| \leq 2.
\end{equation*}
Hence, for all $M \in \N$
\begin{equation*}
     \sup\limits_{v\sim w}\left|\sum\limits_{k=1}^{M}\theta_k\chi_{n_k}(v)-\sum\limits_{k=1}^{M}\theta_k\chi_{n_k}(w)\right|\leq 2.
\end{equation*}
Also notice that for all $M\in \N$,
\begin{equation*}
    \left|\sum\limits_{k=1}^{M}\theta_k\chi_{n_k}(a)\right|\leq 1.
\end{equation*}
Therefore, for all $M\in \N$
\begin{equation*}
    \left\|\sum\limits_{k=1}^{M}\theta_k \chi_{n_k}\right\|_{\Lhat}^{a}=\left|\sum\limits_{k=1}^{M}\theta_k\chi_{n_k}(a)\right|+ \sup\limits_{v\sim w}\left|\sum\limits_{k=1}^{M}\theta_k\chi_{n_k}(v)-\sum\limits_{k=1}^{M}\theta_k\chi_{n_k}(w)\right|\leq 3.
\end{equation*}
Hence we can conclude by Proposition \ref{prop: weak convergence} that $(\chi_n)$ converges weakly to $0$. 

Now observe that for all $n \in \N$ with $n\geq 2$ we have that
\begin{equation}\label{laequ}
    \|\chi_n\|_{\Lhat}^{a}= |\chi_n(a)|+\sup\limits_{v\sim w}|\chi_n(v)-\chi_n(w)|=1
\end{equation}
while for $n=1$ we have that
\begin{equation*}
    \|\chi_n\|_{\Lhat}^{a}= |\chi_n(a)|+\sup\limits_{v\sim w}|\chi_n(v)-\chi_n(w)|=2.
\end{equation*}
Let $K$ be a compact operator on $\Lhat$ (respectively, on $\Lcerohat$). For every fixed $n \in \N$, $n \geq 2$ we have
\begin{align*}
    \| M_\psi - K\| 
    &= \sup \{ \|(M_\psi - K)f\|_{\Lhat}^{a} \mid \|f\|_{\Lhat}^{a}=1 \} \\
    & \geq \sup \{ \|(M_\psi - K) \chi_j \|_{\Lhat}^{a} \mid j \geq n \} \\
    &\geq \sup\{\|M_\psi\chi_j \|_{\Lhat}^{a}- \|K\chi_j\|_{\Lhat}^{a} : j \geq n\}
\end{align*}
Hence we have
\[
    \|M_\psi - K\|\geq\limsup\limits_{n\to \infty}\hspace{.2cm}\left(\|M_\psi\chi_n \|_{\Lhat}^{a}- \|K\chi_n\|_{\Lhat}^{a}\right).
\]
Since $K$ is compact and $\chi_n$ converges weakly to zero, we have
$\lim\limits_{n\to \infty} \|K \chi_n\|_{\Lhat}^{a}=0$, therefore
\[
    \|M_\psi - K\|\geq\limsup\limits_{n\to \infty}\hspace{.2cm}\|M_\psi\chi_n \|_{\Lhat}^{a}.
\]
This implies that 
\begin{equation}\label{ere}
    \|M_\psi\|_{e} 
    \geq \limsup\limits_{n\to \infty}\hspace{.2cm}\|M_\psi\chi_n \|_{\Lhat}^{a}
    = \limsup\limits_{n\to \infty} \hspace{.2cm}\sup\limits_{v\sim w}|\psi(v)\chi_n(v)-\psi(w)\chi_n(w)|=\limsup\limits_{n\to \infty}|\psi(v_n)|.
\end{equation}
But
\[
\limsup\limits_{n\to \infty} \hspace{.2cm}|\psi(v_n)|\\
    =\lim\limits_{m\to \infty} \hspace{.2cm}\sup\bigg\{|\psi(v)|: \dist(a,v)\geq m\bigg\}
    =A(\psi).
\]
Hence, together with inequality \eqref{ere} we can conclude $\|M_\psi\|_{e} \geq A(\psi)$.
\end{proof}

The proof above follows the same lines as \cite[Theorem 8.2]{CoEa}, but differs in some important details. Also, in said paper, Colonna and Easley give a lower bound for the essential norm that also depends on $B(\psi)$. The proof they give does not immediately carry to our setting. We leave the problem of finding a lower bound that depends on $B(\psi)$ for future research.

Next, we will find an upper estimate for the essential norm of $M_\psi$. The proof of this theorem differs from \cite[Theorem 8.3]{CoEa}.

\begin{prop}\label{norma esencial arriba}
If $M_\psi$ is bounded on $\Lhat$ (or equivalently, $\Lcerohat$), then
\begin{equation*}
    \|M_\psi\|_{e} \leq  4\hspace{.03cm}A(\psi) + B(\psi).
\end{equation*}
\end{prop}

\begin{proof}
Fix $n \in \N$ and define the operator $K_n$ on $\Lhat$ as follows
\[ (K_{n}f)(v) =  \begin{cases} f(v), &\text{ if } \dist(a,v) \leq  n,\\
 \frac{2n-\dist(a,v)}{\dist(a,v)}f(v), &\text{ if }   n\leq \dist(a,v)\leq 2n,
 \\ 0,&\text{ if } \dist(a,v) \geq 2n. \end{cases}
\]

% First we will show $K_n$ is a compact operator on $\Lhat$. To do so, consider the following two sets:
% \begin{equation*}
% \begin{split}
%     &A=\{f\in \Lhat : f(v)=0 \text{ for all } v \text{ such that }\dist(a,v)>2n \}\hspace{.2cm}\text{ and}\\
%     &Z=\{ v\in G : \dist(a,v)\leq 2n\}.
% \end{split}
% \end{equation*}
% Notice that $Z$ is a finite set. Now let us denote as $\chi_v$ the characteristic function of the vertex $v$ and consider the set
% \begin{equation*}
%     B=\text{span}\{\chi_v : v \in Z\}.
% \end{equation*}
% Observe that dim$(B)\leq |Z| < \infty$. Also, notice that Im$(K_n)\subseteq A\subseteq B$ which implies that $K_n$ has finite dimensional range.

We first show that $K_n$ is a bounded operator acting on $\Lhat$. Let $x\sim y$. Observe that if $\dist(a,x)\leq n$ and $\dist(a,y)\leq n$, then 
\[
|K_nf(x)-K_nf(y)|=|f(x)-f(y)|\leq \supfvw
\]
and if $\dist(a,x)\geq2n$ and $\dist(a,y)\geq 2n$, then 
\[
|K_nf(x)-K_nf(y)|=\left|0-0\right| \leq \supfvw.
\]
Assume now that $n \leq  \dist(a,x)\leq 2n$ and $n \leq \dist(a,y)\leq 2n$. Observe that
\begin{align*}
|K_nf(x)-K_nf(y)|
=&\left|\frac{2n-\dist(a,x)}{\dist(a,x)}f(x)-\frac{2n-\dist(a,y)}{\dist(a,y)}f(y)\right|\\
\leq & \left|\frac{2n-\dist(a,x)}{\dist(a,x)}f(x)-\frac{2n-\dist(a,y)}{\dist(a,y)}f(x)\right|
        + \left|\frac{2n-\dist(a,y)}{\dist(a,y)}f(x)-\frac{2n-\dist(a,y)}{\dist(a,y)}f(y)\right|\\
=&\left|\frac{2n}{\dist(a,x)}
-\frac{2n}{\dist(a,y)} \right||f(x)|
+ \left|\frac{2n-\dist(a,y)}{\dist(a,y)}\right||f(x)-f(y)| \\
\leq & \left|\frac{2n}{\dist(a,x)\dist(a,y)}
\right| |f(x)| + |f(x)-f(y)| \\
\leq & \left|\frac{2}{\dist(a,x)}
\right| |f(x)| + \supfvw,
\end{align*} 
where the last two inequalities follow from $|\dist(a,x)-\dist(a,y)|\leq 1$, from $0 \leq \frac{2n-\dist(a,y)}{\dist(a,y)} \leq 1$ and from $\dist(a,y)\geq n$.

Since by Proposition \ref{Proposition upper bounds for f} we know that $|f(x)|\leq |f(a)|+\dist(a,x)\sup\limits_{v\sim w}|f(v)-f(w)|$. Then we obtain
\begin{equation*}
    \frac{2}{\dist(a,x)} |f(x)| \leq \frac{2}{\dist(a,x)} |f(a)| +2\sup\limits_{v\sim w}|f(v)-f(w)|\leq  2|f(a)| +2\sup\limits_{v\sim w}|f(v)-f(w)|.
\end{equation*}
Combining this with the displayed inequality above we obtain
\[
|K_nf(x)-K_nf(y)| \leq 2 | f(a) | + 3 \supfvw.
\]
Hence we have
\begin{equation}\label{norm_Kn}
  \| K_n f \|_{\Lhat}^{a} =  |K_nf(a)|+ \sup_{v \sim w} |K_nf(v)-K_nf(w)| = |f(a)|+ \sup_{v \sim w} |K_nf(v)-K_nf(w)|\leq 3\fnorm
\end{equation}
thus $K_n$ is a bounded operator on $\Lhat$. Since clearly $K_n$ has finite dimensional range, it follow that $K_n$ is compact operator.

Lastly, for $f \in \Lhat$, we will compute an estimate for $\|M_\psi (I-K_n)f\|_{\Lhat}^{a}$. First, observe that $(I-K_n)f(v)=0$ for all $v \in G$ such that $\dist(a,v)\leq n$, where we write $(I-K_n)f(v):=((I-K_n)f)(v)$ for simplicity. We then have
\begin{align*}
\|M_\psi (I-K_n)f\|_{\Lhat}^{a}
       &=|\psi(a)(I-K_n)f(a)|+\sup\limits_{v\sim w}|\psi(v)(I-K_n)f(v)-\psi(w)(I-K_n)f(w)|\\
       &=\sup\left\{ |\psi(v)(I-K_n)f(v)-\psi(w)(I-K_n)f(w)| :  v\sim w,\hspace{.2cm} \dist(a,v), \dist(a,w)\geq n \right\}.
       \end{align*}
Now, if $v \sim w$ and $\dist(a,v), \dist(a,w)\geq n$ we have
\begin{equation}\label{psiKnf}
\begin{split}
|\psi(v) (I-K_n)f(v)-\psi(w)(I-K_n)f(w)| 
\leq & \left|(\psi(v)-\psi(w))(I-K_n)f(v) \right| \\
& \quad +\left|\psi(w)\left((I-K_n)f(v) -(I-K_n)f(w)\right)\right|. 
\end{split}
\end{equation}

Observe that $|(I-K_n)f(v)|\leq |f(v)|$, and hence from Proposition \ref{Proposition upper bounds for f} we have
\begin{equation}\label{psivwIKn}
\left|(\psi(v)-\psi(w))(I-K_n)f(v) \right|\leq |\psi(v)-\psi(w)| \dist(a,v) \fnorm. 
\end{equation}

Also, observe that, by inequality \eqref{norm_Kn}, for $v \sim w$ we have
\[
\left|(I-K_n)f(v)-(I-K_n)f(w) \right|
\leq |f(v)-f(w) | + \left|K_nf(v)- K_nf(w) \right| \leq \fnorm + \|K_n f \|_{\Lhat}^{a} \leq 4 \fnorm.
\]
Hence, for $v \sim w$ with $\dist(a,v), \dist(a,w) \geq n$ we have
\begin{equation}\label{psiIKn4}
|\psi(w)||(I-K_n)f(v)-(I-K_n)f(w)| \leq  4\fnorm \sup\left\{|\psi(v)|:\dist(a,v)\geq n\right\}.
\end{equation}
Using inequality \eqref{psiKnf}, and inequalities \eqref{psivwIKn} and \eqref{psiIKn4} we have that
\begin{align*}
\|M_\psi (I-K_n)f\|_{\Lhat}^{a}
\leq & \sup\left\{ |\psi(v)(I-K_n)f(v)-\psi(w)(I-K_n)f(w)| :  v\sim w,\hspace{.2cm} \dist(a,v), \dist(a,w)\geq n \right\} \\
\leq & \fnorm \ \sup\{ \dist(a,v) |\psi(v)-\psi(w)|  \mid v \sim w, \dist(a,v), \dist(a,w) \geq n \} \\
& \quad +  4\fnorm \  \sup\left\{|\psi(v)|: \dist(a,v)\geq n\right\},
\end{align*}
and hence
\[
\|M_\psi (I-K_n)\| \leq \sup\{ \dist(a,v) |\psi(v)-\psi(w)|  \mid v \sim w, \dist(a,v), \dist(a,w) \geq n \} + 4 \  \sup\left\{|\psi(v)|: \dist(a,v)\geq n\right\},
\]

To conclude the proof, notice that, since $M_\psi K_n$ is a compact operator, we have
\begin{align*}        
\|M_\psi\|_{e}
&=\inf\{\|M_\psi-K\| : K \text{ compact operator on } \Lhat\}\\
&\leq \inf\{\|M_\psi-M_\psi K_n\| : n \in \N\}\\
&\leq\inf\bigg\{\sup\left\{\dist(a,v)|\psi(v)-\psi(w)| \mid v\sim w, \dist(a,v), \dist(a,w)\geq n\right\}\\
&\qquad \qquad + 4\sup\left\{|\psi(w)| \mid \dist(a,v)\geq n\right\} : n \in \N\bigg\}\\
&= \lim\limits_{n\to \infty }\sup\bigg\{\dist(a,v)|\psi(v)-\psi(w)|\mid v\sim w,  \dist(a,v), \dist(a,w) \geq n\bigg\}\\
&\qquad \qquad +\lim\limits_{n\to \infty }4\sup\bigg\{|\psi(w)|\mid \dist(a,v)\geq n\bigg\}\\
&=B(\psi)+4 A(\psi).\qedhere
\end{align*}
\end{proof}

\section{Isometries}
\label{chap: Isometries}

Observe that if $\lambda \in \C$ satisfies that $|\lambda|=1$, then $M_\psi$ is an isometry if $\psi(v)=\lambda$ for every $v \in G$. As it turns out, there are no other isometric multiplication operators on $\Lhat$ or on $\Lcerohat$.

\begin{prop}
The only isometric multiplication operators $M_\psi$ on $\Lhat$ or $\Lcerohat$ are those for which $\psi$ is constant of modulus one.
\end{prop}
\begin{proof}
Assume $M_\psi$ is an isometry on $\Lhat$ or $\Lcerohat$. Since $\|1\|_{\Lhat}^{a}=1$, 
\begin{equation}\label{aaa}
   \|M_\psi 1 \|_{\Lhat}^{a}=\psinorm = 1.
\end{equation}
On the other hand, let us define the function $g:\Lhat\longrightarrow \C$ as $g = \frac{1}{2}\chi_{\{a\}}$ where $\chi_{\{a\}}$ is the characteristic function of the vertex $a$. Notice that
\[
\gnorm = |g(a)|+\sup\limits_{v\sim w}|g(v)-g(w)|=\frac{1}{2}+\frac{1}{2}=1.
\]
Finally observe that
\begin{equation}\label{si}
  \|\psi g\|_{\Lhat}^{a} = |\psi(a)g(a)|+\sup\limits_{v\sim w}|\psi(v)g(v)-\psi(w)g(w)|
  =2|\psi(a)g(a)|
  =|\psi(a)|.
\end{equation}

Then, since $M_\psi$ is an isometry on $\Lhat$ and $\|g\|_{\Lhat}^{a}=1$ it follows that $\|\psi g\|_{\Lhat}^{a}=1$. By equation \eqref{si} that $\|\psi g\|_{\Lhat}^{a}=|\psi(a)|$ and hence $|\psi(a)|=1$. Furthermore, from equation \eqref{aaa} we know that $\|\psi\|_{\Lhat}^{a}=1$ and so we can see $\sup\limits_{v\sim w}|\psi(v)-\psi(w)|=0$. Therefore we can conclude $\psi$ is a constant function of modulus one, as desired.
\end{proof}

\section*{Statements and declarations}

\subsection*{Competing interests}
The authors have no interests to declare that are relevant to the content of this article.

\subsection*{Ethics declaration}
The authors have maintened the highest ethical standards in the research presented in this article.

\subsection*{Data availability}
No data was used for the research presented in this article.

\bibliographystyle{amsplain}

\bibliography{biblio}{}

@article {CoEa,
    AUTHOR = {Colonna, Flavia and Easley, Glenn R.},
     TITLE = {Multiplication operators on the {L}ipschitz space of a tree},
   JOURNAL = {Integral Equations Operator Theory},
  FJOURNAL = {Integral Equations and Operator Theory},
    VOLUME = {68},
      YEAR = {2010},
    NUMBER = {3},
     PAGES = {391--411},
      ISSN = {0378-620X,1420-8989},
   MRCLASS = {47B38 (05C90)},
  MRNUMBER = {2735443},
       DOI = {10.1007/s00020-010-1824-5},
       URL = {https://doi.org/10.1007/s00020-010-1824-5},
}

@article {CoLo,
    AUTHOR = {Colonna, Flavia and Locke, Rachel E.},
     TITLE = {Multiplication operators on the space of an infinite graph},
   JOURNAL = {preprint},
  FJOURNAL = {},
    VOLUME = {},
      YEAR = {2023},
    NUMBER = {},
     PAGES = {},
      ISSN = {},
       DOI = {},
       URL = {},
}

@article {DuRoSh,
    AUTHOR = {Duren, P. L. and Romberg, B. W. and Shields, A. L.},
     TITLE = {Linear functionals on {$H\sp{p}$} spaces with {$0<p<1$}},
   JOURNAL = {J. Reine Angew. Math.},
  FJOURNAL = {Journal f\"ur die Reine und Angewandte Mathematik. [Crelle's
              Journal]},
    VOLUME = {238},
      YEAR = {1969},
     PAGES = {32--60},
      ISSN = {0075-4102,1435-5345},
   MRCLASS = {46.30},
  MRNUMBER = {259579},
MRREVIEWER = {D.\ Sarason},
}

@book {ReSi,
    AUTHOR = {Reed, Michael and Simon, Barry},
     TITLE = {Methods of modern mathematical physics. {I}. {F}unctional
              analysis},
 PUBLISHER = {Academic Press, New York-London},
      YEAR = {1972},
     PAGES = {xvii+325},
   MRCLASS = {47-02 (81.47)},
  MRNUMBER = {493419},
MRREVIEWER = {P.\ R.\ Chernoff},
}

@incollection {CoCo1,
    AUTHOR = {Cohen, Joel M. and Colonna, Flavia},
     TITLE = {The {B}loch space of a homogeneous tree},
      NOTE = {Papers in honor of Jos\'e{} Adem},
   JOURNAL = {Bol. Soc. Mat. Mexicana (2)},
  FJOURNAL = {Bolet\'in de la Sociedad Matem\'atica Mexicana. Segunda Serie},
    VOLUME = {37},
      YEAR = {1992},
    NUMBER = {1-2},
     PAGES = {63--82},
   MRCLASS = {31C20 (30H05 43A15)},
  MRNUMBER = {1317563},
MRREVIEWER = {Martha\ B\u anulescu},
}

@article {CoCo2,
    AUTHOR = {Cohen, Joel M. and Colonna, Flavia},
     TITLE = {Embeddings of trees in the hyperbolic disk},
   JOURNAL = {Complex Variables Theory Appl.},
  FJOURNAL = {Complex Variables. Theory and Application. An International
              Journal},
    VOLUME = {24},
      YEAR = {1994},
    NUMBER = {3-4},
     PAGES = {311--335},
      ISSN = {0278-1077,1563-5066},
   MRCLASS = {30F35 (05C05 30E05 31C05)},
  MRNUMBER = {1270321},
MRREVIEWER = {Vadim\ A.\ Ka\u imanovich},
       DOI = {10.1080/17476939408814723},
       URL = {https://doi.org/10.1080/17476939408814723},
}

@incollection {Cartier1,
    AUTHOR = {Cartier, P.},
     TITLE = {Fonctions harmoniques sur un arbre},
 BOOKTITLE = {Symposia {M}athematica, {V}ol. {IX} ({C}onvegno di {C}alcolo
              delle {P}robabilit\`a{} \& {C}onvegno di {T}eoria della
              {T}urbolenza, {INDAM}, {R}ome, 1971)},
     PAGES = {203--270},
 PUBLISHER = {Academic Press, London-New York},
      YEAR = {1972},
   MRCLASS = {60J50 (22E50 43A85)},
  MRNUMBER = {353467},
MRREVIEWER = {J.-P.\ Conze},
}

@incollection {Cartier2,
    AUTHOR = {Cartier, Pierre},
     TITLE = {G\'eom\'etrie et analyse sur les arbres},
 BOOKTITLE = {S\'eminaire {B}ourbaki, 24\`eme ann\'ee (1971/1972)},
    SERIES = {Lecture Notes in Math.},
    VOLUME = {Vol. 317},
     PAGES = {Exp. No. 407, pp. 123--140},
 PUBLISHER = {Springer, Berlin-New York},
      YEAR = {1973},
   MRCLASS = {22E50},
  MRNUMBER = {425032},
MRREVIEWER = {G.\ I.\ Ol\cprime shanski\u i},
}

@article {AlCoEa,
    AUTHOR = {Allen, Robert F. and Colonna, Flavia and Easley, Glenn R.},
     TITLE = {Composition operators on the {L}ipschitz space of a tree},
   JOURNAL = {Mediterr. J. Math.},
  FJOURNAL = {Mediterranean Journal of Mathematics},
    VOLUME = {11},
      YEAR = {2014},
    NUMBER = {1},
     PAGES = {97--108},
      ISSN = {1660-5446,1660-5454},
   MRCLASS = {47B38 (05C05)},
  MRNUMBER = {3160615},
MRREVIEWER = {George\ Stacey\ Staples},
       DOI = {10.1007/s00009-013-0308-7},
       URL = {https://doi.org/10.1007/s00009-013-0308-7},
}

@book {Locke,
    AUTHOR = {Locke, Rachel E.},
     TITLE = {Multiplication {O}perators in {D}iscrete {S}ettings of an
              {I}nfinite {G}raph and the {D}iscrete {Z}ygmund {S}pace},
      NOTE = {Thesis (Ph.D.)--George Mason University},
 PUBLISHER = {ProQuest LLC, Ann Arbor, MI},
      YEAR = {2016},
     PAGES = {145},
      ISBN = {978-1339-89905-3},
   MRCLASS = {99-05},
  MRNUMBER = {3527370},
       URL =
              {http://gateway.proquest.com/openurl?url_ver=Z39.88-2004&rft_val_fmt=info:ofi/fmt:kev:mtx:dissertation&res_dat=xri:pqm&rft_dat=xri:pqdiss:10132074},
}

\end{document}